\documentclass[a4paper,11pt,reqno]{amsart}
\usepackage[T1]{fontenc}
\usepackage[utf8]{inputenc}
\usepackage[margin=1.0in]{geometry}
\usepackage{lmodern}
\usepackage[english]{babel}
\usepackage{graphicx}
\usepackage{amsmath}
\usepackage{amsfonts}
\usepackage{amssymb} 
\usepackage{amsthm}
\usepackage{bbm}
\usepackage{bm} 
\usepackage{lipsum}
\usepackage{xcolor}
\usepackage{hyperref}
\allowdisplaybreaks

\theoremstyle{plain}
\newtheorem{thm}{Theorem}[section]
\newtheorem{lem}[thm]{Lemma}
\newtheorem{prop}[thm]{Proposition}

\newtheorem{defn}[thm]{Definition}
\newtheorem{rem}[thm]{Remark}
\newtheorem{ass}{Assumption}

\theoremstyle{definition}

\numberwithin{equation}{section}

\def \be {\begin{equation}}
	\def \ee {\end{equation}}

\def \E {\mathbb{E}}
\def \P {\mathbb{P}}
\def \R {\mathbb{R}}
\def \Law {\mathrm{Law}}
\def \T {\mathbb{T}}
\def \dn {\Delta x_n} 
\def \bm1 {\mathbbm{1}}

\renewcommand{\phi}{\varphi}
\renewcommand{\epsilon}{\varepsilon}
\renewcommand{\tilde}{\widetilde}
\renewcommand{\hat}{\widehat}
\renewcommand{\bar}{\overline}

\begin{document}

	\title{Mean field games master equations: from discrete to continuous state space}
	
	\author{Charles Bertucci and Alekos Cecchin}

	\address[C. Bertucci]
	{\newline \indent Centre de Math\'ematiques Appliqu\'ees, CNRS, UMR 7641, \'Ecole Polytechnique
		\newline
		\indent Route de Saclay, 
		91128 Palaiseau Cedex, France }
	\email{charles.bertucci@polytechnique.edu}
	\address[A. Cecchin]
	{\newline \indent Dipartimento di Matematica ``Tullio Levi-Civita'', Università di Padova
		\newline
		\indent Via Trieste 63, 35121 Padova, Italy}
	\email{alekos.cecchin@unipd.it}
	\thanks{C. Bertucci acknowledges financial support of chair FDD of Institut Louis Bachelier. A. Cecchin benefited from the support of LABEX Louis Bachelier - project ANR-11-LABX-0019, ECOREES ANR Project, FDD Chair and Joint Research Initiative
		FiME; he was also partially supported by the INdAM-GNAMPA Project 2023 ``Stochastic mean field models: analysis and applications''}

	\date{September 11, 2023}
	

	\begin{abstract}
		%
		This paper studies the convergence of mean field games with finite state space to mean field games with a continuous state space. We examine a space discretization of a diffusive dynamics, 
		which is reminiscent of the Markov chain approximation method in stochasctic control, but also of finite difference numerical schemes; time remains continuous in the discretization, and the time horizon is arbitrarily long. 
		We are mainly interested in the convergence of the solution of the associated master equations as the number of states tends to infinity. We present two approaches, to treat the case without or with common noise, both under monotonicity assumptions. The first one uses the system of characteristics of the master equation, which is the MFG system,
		to establish a convergence rate for the master equations without common noise and the associated optimal trajectories, both in case there is a smooth solution to the limit master equation and in case there is not. 
		The second approach relies on the notion of monotone solutions introduced by \cite{bertucci2021monotone, bertucci2021monotone2}. 
		In the presence of common noise, we show convergence of the master equations, with a convergence rate if the limit master equation is smooth, otherwise by compactness arguments. 
	\end{abstract}
	
	\maketitle
	

	\section{Introduction}
	This paper is interested in the convergence of value functions of mean field games (MFG for short) in finite state space when the number of states tends to infinity. We show that if the MFG in finite state space is a suitable discretization of a continuous MFG, the value, i.e. the solution of the master equation, in finite state space converges toward the value of the MFG in continuous state space when the number of states tends to infinity, and we provide a rate for such convergence.
	
	\subsection{General introduction}
	MFG are differential games involving non-atomic players which interact only through mean field terms. A general mathematical study of such games started with \cite{lasry2007mean,lions2007cours} and independently in \cite{hcm}. We refer to the books \cite{carmona2018probabilistic,cardaliaguet2019master} for a more complete presentation of the theory, and also to the lecture notes \cite{notescetraro}. Several properties of those games being understood by now, let us stress the two properties which are the most helpful to understand the following. The first one is that Nash equilibria of the game can be characterized in terms of a system of differential equations which may be stochastic or not, depending on the nature of the game, and which are ordinary differential equations (ODE for short) if the state space of the players is finite or partial differential equations (PDE for short) if the state space of the players is continuous. Such a characterization of the Nash equilibria is called the MFG system. The second main aspect of MFG is that an adversarial regime can be identified. In this so-called monotone regime, there is always a unique Nash equilibrium in the MFG. This property allows to define a concept of value in this situation. This value is the solution of a PDE called the master equation, which is set on a finite dimensional space (the simplex in $\R^n$) if the state space of the players is finite and on an infinite dimensional space (the space of probability measures) if the state space of the players is continuous. Notably, the MFG system represents the system of characteristics of the master equation.
	
	MFG master equations have attracted quite a lot of attention in the last years. 
	First in the continuous state space, we mention the main contribution  
	\cite{cardaliaguet2019master} for classical solutions (see also \cite{CCD2014, zhang2021}) and 
	\cite{mou2020,gangbo2021,bertucci2021monotone2,cardaliaguet2021weak,mou2022, cecchindelarue2022} for various definitions of weak solutions. In general, to have classical solutions to the master equation,  which is a PDE in the space of probability measures, for arbitrary time horizon, the cost coefficients are required to be differentiable with respect to the measure argument and monotone (see however the recent preprint \cite{mou2022}), otherwise weak solutions have to be considered, which are at least continuous in the measure argument if monotonicity holds. 
	In the finite state space, classical solutions are considered in \cite{bayraktarcohen,cecchin2019b}, without common noise, and  in  \cite{bertucci2019some,bayraktar2021,delaruecetraro} with various forms of common noise. Some definitions of weak solution are given in  \cite{cecchin2019a,cecchin2021,bertucci2021monotone}.

	
	\subsection{Bibliographical comments}
	The first numerical schemes for mean field games were proposed in \cite{achdou2010mean,achdou2012mean,achdou2013mean}, based on finite difference numerical methods for PDEs. 
	Several other methods have been studied: as an incomplete list, 
	we mention  \cite{benamou2015augmented} for an augmented Lagrangian method,  \cite{almulla2017two} for gradient flow, 
	\cite{briceno2018proximal} for a proximal method, \cite{bertucci2020uzawa} for the use of Uzawa's algorithm, \cite{CCD2019} for a probabilistic method based on the forward-backward system of SDEs, the survey  \cite{laurierebis}  for methods based on machine learning, and \cite{gomes_numerics} for a method for finite state mean field games; we mention also the recent survey \cite{lauriere}. 
	The discretization we analyze in the paper has the advantage of being itself a mean field game, on a finite state space. 
	The convergence of discrete-time finite state MFG toward ones with continuous state space has been studied in \cite{haeedinkanloo}, in the case of deterministic dynamics without idiosyncratic or common noise.
	Their proof of convergence relies on probabilistic compactness  and weak convergence arguments and on the probabilistic representation of the MFG. 

	\subsection{Main results of the paper} 
	The discretization we study in the paper is the natural one, both from a PDE and a probabilistic perspective. At the PDE level, it is almost equal to the finite difference approximation studied in \cite{achdou2010mean}. We exploit also the probabilistic interpretation, which turns out to be close to the scheme studied in \cite{haeedinkanloo} and is based on the Markov chain approximation method for stochastic control problems introduced by Kushner \cite{kushner}. As already mentioned, such discretization has the advantage of being itself a MFG over a finite state space, of the type first analyzed in \cite{gomes}.  
	One of the main differences with the aforementioned work is that the time remains continuous in the discretized model. 
	We consider dynamics on the one dimensional torus and with a non-degenerate idiosyncratic noise; we always assume monotonicity of the cost coefficients and thus consider an arbitrary time horizon.  
	We remark that the focus on a one dimensional state space with periodic boundary condition is mainly to simplify the already heavy notations. We leave to the interested reader the generalization to a higher dimensional state space and indicate along the paper the arguments which do not immediately extend to such a situation. 
	
	The main difference with the other works on numerical methods for MFGs is that our strategy to show the convergence is based on the master equations of the discrete and continuos MFGs. One of the main results is to provide a convergence rate for the approximation, which is a new result for numerical methods for MFGs, to the best of our knowledge.
	\footnote{We remark, however, that several months after posting the preprint of the present paper on arXiv, the preprint \cite{pfeiffer} appeard on arXiv, in which the authors establish a  rate for the convergence of a $\theta$-scheme to the MFG system without common noise, with different methods.}
	Our analysis is in fact a theoretical study of the convergence of finite towards continuous mean field games and, even though we do not solve the curse of dimensionality, our results permit to approximate the continuous state master equation, which is a PDE set on the infinite dimensional space of probability measure, with more tractable PDEs set on finite dimensional spaces whose dimension grows, and we establish a rate for such convergence.

	The approach of this paper is twofold. In a first time, we study MFG without common noise. We first consider the case in which there exists  classical solutions to the  master equations, and establish a rate of convergence for both the value of the MFG (the master equation) and the optimal trajectories (Thm. \ref{thm:conv:Un} and \ref{thm:conv:traj}). The method employed to prove these results is inspired by the stability argument for forward-backward systems also used in \cite{cardaliaguet2019master}. 
	We then present a less restrictive approach, without considering directly the master equation which may not have a smooth solution. In this case, we use just the system of characteristics (the MFG system) and the monotonicity and prove a convergence rate for both the value of the MFG and the trajectories at the equilibrium (Thm. \ref{thm:conv5}); indeed, we establish a rate for the convergence of the MFG systems, by exploiting mainly their stability under monotonicity. Notably, the convergence rate we obtain is worse in case there is no smooth solution to the master equation. We remark also that the non-degeneracy of the independent noise (in other words, the presence of the laplacian) is crucial to obtain a convergence rate. 
	Along the way, we show a convergence rate for a Markov chain approximation of a diffusion, which we believe might be of independent interest and thus is presented in the Appendix \ref{appendixA} (Prop. \ref{lem:7}).  
	
	In a second time, we study MFG with the type of common noise introduced in \cite{bertucci2019some}, in the monotone regime, which basically produces common jumps of the whole population. We first prove  a compactness result on the master equation of the discrete model, which allows to establish the convergence of the value of the MFG, which is the master equation {(Thm. \ref{thm:main:mono})}. With such type of common noise,  we do not use the system of characteristics because it might be stochastic and with jumps and thus more difficult to treat. Instead, this part on MFG with common noise relies on an intrinsic study of the master equation through the concept of monotone solutions introduced in \cite{bertucci2021monotone,bertucci2021monotone2} and used in \cite{cardaliaguet2021weak}.  This concept enables us to work with solution of the master equations which are merely continuous in the finite state space case and only continuous with respect to the measure variable in the continuous state space limit. The aforementioned convergence result can be seen as an illustration of the stability of these monotone solutions. Since the method of the proof of this result is based on a compactness argument, we do not get in general a convergence rate in this case. 
	Finally, we show that if a classical solution of the limit problem exists, then a rate of convergence for the value can be proved also in this case (Thm. \ref{thm20}), by exploiting a variation of the method introduced by \cite{lions2007cours}. 
	
	\subsection{Organization of the paper} The rest of the paper is organized as follows. In Section \ref{sec:2}, we present first the continuous state MFG in \S \ref{S:2.1} and then its discretization, which is the finite state MFG, in \S \ref{sec:2.2}, both from a probabilistic and PDE point of view, together with their master equations. The standing Assumptions are stated in \S \ref{S:2.2}. Section \ref{sec:3} is devoted to the study of convergence in the absence of common noise: the approach with a smooth solution to the master equation is in \S \ref{sec:3.1}, while the approach based on the MFG system is in \S \ref{sec:3.2}. Section \ref{sec:4} studies the convergence for MFGs with common noise: the compactness estimates is presented in \S \ref{sec:4.4} and the convergence of monotone solutions is in \S \ref{sec:4.5}, while the convergence for smooth solutions is in \S \ref{sec:4.6}; the other subsections contain auxiliary results and remarks on monotone solutions. Finally, Appendix \ref{appendixA} contains the result about the convergence rate  for a Markov chain approximation of a diffusion.

	\subsection{Notation}
	\begin{itemize}
		\item $\langle \cdot, \cdot \rangle$ stands for either the usual scalar product between two element of $\R^d$ or for the extension of the $L^2$ scalar product for functional spaces in duality, depending on the context.
		\item The unit circle is denoted by $\mathbb{T}$.
		\item The set of Borel measures on $E$ is denoted by $\mathcal{M}(E)$ whereas $\mathcal{P}(E)$ stands for the set of probability measures on $E$.
		\item The usual norms on the H\"older spaces $\mathcal{C}^{k+\gamma}(\T)$ are denoted by $\|\cdot\|_{k+ \gamma}$, while the norms on $\mathcal{C}^{k'+\beta, k+\gamma}([0,T]\times\T)$ are denoted $\|\cdot\|_{k'+\beta, k+ \gamma}$, {for $k,k' \in \mathbb{N}$ and $0<\beta,\gamma<1$.}
		\item For a function $U : \mathcal{P}(\mathbb{T}) \to \R$, when it is defined we denote, for $m\in \mathcal{P}(\mathbb{T})$ and $x\in\T$,
		\be
		\label{delta_m_der}
		\frac{\delta U}{\delta m}(m,x) = \lim_{\theta \to 0}\frac{ U((1-\theta)m + \theta \delta_x) - U(m)}{\theta}.
		\ee
		\item For $\mu,\nu \in \mathcal{P}(E)$, with $E$ a metric space, we denote by $W_1$ the Monge-Kantorovich distance between $\mu$ and $\nu$, {which is also called the 1-Wasserstein distance.} 
		\item We fix a filtered probability space $(\Omega, \mathbb{F}= (\mathcal{F}_t)_{0\leq t\leq T}, \mathbb{P})$ satisfying the usual conditions, large enough to contain all the processes we will introduce. All SDEs will have indeed pathwise strong solutions. The law of a random variable is denoted by 
		$\Law(\xi) = \P \circ \xi^{-1}$.  
		\item  $\mathcal{D}([0,T], \T)$ is the space of c\`adl\`ag functions endowed with the Skorokhod $J_1$ topology. Convergence of processes in law is meant as usual on this space. 
	\end{itemize}
	
	\section{The continuous and discrete models} 
	\label{sec:2}
	
	
	In this section we introduce first the model at interest in the limit of an infinite number of states, in the case without common noise, and then its space discretization. As already mentioned, we shall focus on a one dimensional state space with periodic boundary condition. 
	
	\subsection{The limit MFG model} 
	\label{S:2.1}
	We first consider the master equation of unknown $U:[0,T] \times \mathbb{T} \times \mathcal{P}(\T) \rightarrow \R$, in dimension 1 on the torus:
	\be
	\label{master}
	\begin{split}
		&-\partial_t U - \sigma\partial_{xx} U +  H(x,\partial_x U) 
		-  \left\langle \sigma \partial_{xx} m + \partial_x(\partial_pH(\cdot,\partial_x U) m), \frac{\delta U}{\delta m}(t,x,m, \cdot) \right\rangle = f(m)(x) \\
		& U(T,x,m) = g(m)(x),
	\end{split}
	\ee
	where $\sigma > 0$, $H : \mathbb{T} \times \mathbb{R} \to \mathbb{R}, f,g : \mathcal{P}(\mathbb{T}) \to \mathcal{C}(\mathbb{T})$ and $T > 0$ are the data of the problem. This equation corresponds to the following MFG. The dynamics of a player is
	\be 
	\label{dyn}
	dX_t = \alpha(t,X_t) dt + \sqrt{2\sigma} d {B_t},
	\ee 
	where $\alpha$ is its closed-loop control (assumed to be bounded) and {$(B_t)_{t \geq 0}$} is a Brownian motion on $(\Omega, \mathbb{F}, \P)$. Given an anticipation $(\mu_t)_{t \in [0,T]}$ on the repartition of the players in the state space, the expected cost of this player is given by
	\be 
	\label{cost}
	J(\alpha,(\mu_t)_{t \geq 0}) = \E\left[ \int_0^T L(X_t,\alpha(t,X_t))+ f(\mu_t)(X_t) dt +g(\mu_T)(X_T) \right],
	\ee 
	where $L$ is a cost function such that its Legendre transform $L^*$ with respect to its second argument is equal to the Hamiltonian $H(x,p)$. 
	
	We recall that a solution of the MFG (or Nash equilibrium) is a couple $(\alpha, \mu)$ such that $J(\alpha, \mu) = \inf_\beta J(\beta, \mu)$ and $\Law(X_t) = \mu_t$, where $X$ is the optimal process given by $\alpha$. 
	Nash equilibria of the MFG can be characterized through the MFG system 
	\be 
	\label{MFGsystem}
	\begin{cases}
		-\partial_t u- \sigma \partial_{xx} u + H(x,\partial_x u) = f( \mu_t)(x) &\\
		\partial_t \mu - \sigma\partial_{xx} \mu - \partial_x ( \partial_pH(x,\partial_xu)\mu) =0 &\\
		u(T,x)= g(\mu_T)(x) \qquad \mu_0 =m_0.
	\end{cases}
	\ee 
	and the optimal control is $\alpha(t,x)= -\partial_p H(x, \partial_x u(t,x))$,
	where the backward equation is the HJB equation for the value function $u$ of the control problem \eqref{dyn}-\eqref{cost}, while the forward equation is the KFP equation for the consistency condition. We remark that \eqref{MFGsystem} represents the system of characteristics of \eqref{master}, in the sense that $u(t,x)=U(t,x,\mu_t)$ if the equations are well-posed.  
	{A \emph{classical solution} to the master equation \eqref{master} is defined as a function $U:[0,T] \times \mathbb{T} \times \mathcal{P}(\T) \rightarrow \R$ such that the equation is satisfied and $U$ and all the derivatives appearing in the equation are continuous in their arguments $(t,x,m)$, the set $\mathcal{P}(\T)$ being equipped with the 1-Wasserstein distance $W_1$. More precisely, after integration by parts, the derivatives appearing in the equation are $\partial_x U$, $\partial_{xx} U$, $\frac{\delta U}{\delta m}$, $\partial_y  \frac{\delta U}{\delta m}$ and $\partial_{yy}  \frac{\delta U}{\delta m}$.}
	
	\subsection{Assumptions} 
	\label{S:2.2}
	
	We state the assumptions which are in force throughout the paper. 
	We assume that the couplings $f,g : \mathcal{M}(\mathbb{T}) \to \mathcal{C}^{\gamma}(\mathbb{T})$ 
	%
	are monotone, i.e. 
	\be 
	\label{mono}
	\begin{split}
		\int_\T (f(x,m) -f(x, \tilde{m})) (m- \tilde{m})(dx) &\geq 0  \qquad \forall m, \tilde{m} \in \mathcal{M}(\T) \\
		\int_\T (g(x,m) -g(x, \tilde{m})) (m- \tilde{m})(dx) &\geq 0 \qquad \forall m, \tilde{m} \in \mathcal{M}(\T),
	\end{split} 
	\ee
	and also that they are $W_1$-Lipschitz continuous in $m$ (uniformly in $x$), $f$ is Lipschitz also in $x$ (uniformly in $m$) and $g$ is (valued and) bounded in $\mathcal{C}^{2+\gamma}(\T)$, uniformly in $m$, for a $\gamma \in (0,1)$.

	The Hamiltonian $H(x,p)$ is $\mathcal{C}^2$, uniformly convex in $p$ on all compact sets. 
	The duration of the game $T> 0$ is arbitrarily long but fixed.

	\subsection{The discrete MFG model}
	\label{sec:2.2}
	
	We construct approximations on the previous model with finite state and continuous time. The dynamics of the underlying Markov chain has the peculiarity that it jumps either right or left, with the convention that at the boundary of the torus it jumps on the other side. 
	For any $n$, we consider then the $n$ states $S^n=\{ x_1^n, \dots, x_n^n\}=\{ 1/n, \dots, 1=0\}$ with mutual distance $\Delta x_n = 1/n$.
	The discretization we study is the natural one, from a control or MFG perspective, see for instance the seminal papers \cite{achdou2010mean,achdou2012mean} and \cite{haeedinkanloo}.

	Let us state first the probabilistic interpretation of the discrete model by means of controlled Markov chains. 
	We assume to control the jump rate on the right and on the left by means of functions  denoted  $\alpha_+^n, \alpha_-^n :[0,T] \times S^n\rightarrow [0,+\infty)$; the Markov chain $X^n$ then satisfies 
	\be
	\label{rate:n}
	\begin{split}
		&\P( X^n_{t+\Delta t} = x^n_{i+1} | X^n_t = x^n_i) = \left(\frac{\alpha_+^n(t,x^n_i)}{\dn} + \frac{\sigma}{\dn^2} \right) \Delta t +o(\Delta t), \\
		& \P( X^n_{t+\Delta t} = x^n_{i-1} | X^n_t = x^n_i) = 
		\left(\frac{\alpha_-^n(t,x^n_i)}{\dn} + \frac{\sigma}{\dn^2} \right)\Delta t +o(\Delta t), 
	\end{split}
	\ee
	with the convention that $x_{n+1}^n =x_1^n =1/n$ and $x_0^n =x_n^n = 1$.
	Given anticipations $(\mu^n_t)_{t \geq 0}$ on the repartition of players in $S^n$ the cost is given by
	\be 
	\label{cost:n} 
	J^n(\alpha_{\pm}^n, \mu^n) = \E\left[ \int_0^T L(X^n_t,\alpha_+^n(t,X^n_t)) +
	L(X^n_t,-\alpha_-^n(t,X^n_t))- L(X^n_t,0) + f(\mu^n_t)(X^n_t) dt +g(\mu^n_T)(X^n_t) \right].
	\ee
	A solution of the MFG is still a couple $(\alpha_{\pm}^n, \mu^n)$ such that $\alpha_{\pm}^n$ is optimal for $\mu^n$ fixed and 
	$\mu^n_t = \Law(X^n_t)$, where $X^n$ is the optimal process.
	
	For a function $u: S^n \rightarrow \R$ we denote the right and left  first order $n$ finite difference by
	\be 
	\Delta_+^n u(x) = \frac{u(x+\dn)-u(x)}{\dn},
	\qquad \Delta_-^n u(x) = \frac{u(x-\dn)-u(x)}{\dn},
	\ee 
	and the second order finite difference
	\be 
	\Delta_2^n u(x) =\frac{u(x+\dn)-2 u(x)+u(x-\dn)}{\dn^2} .
	\ee
	We remark that if $u : \mathbb{T}\to \R$ is smooth, then $\lim_{n\rightarrow \infty} \Delta_{\pm}^n u(x)  =\pm\partial_x u(x)$ and $\lim_{n\rightarrow \infty} \Delta_2^n u(x)  =\partial_x^2 u(x)$. The optimization provides the discrete HJB equation
	\be 
	\label{HJB:n} 
	-\frac{d}{dt} u^n +H_{\uparrow}(x,\Delta_+^n u^n(x)) +H_{\downarrow}(x,-\Delta_-^n u^n(x)) - \sigma\Delta_2^n u^n(x) =f(\mu^n_t)(x) , \qquad x\in S^n,
	\ee 
	where 
	\be
	\label{Hn}
	H_{\uparrow}(x,p) := -\inf_{\alpha \geq 0} \{L(x,\alpha) + \alpha p\}; \qquad  H_{\downarrow}(x,p) := -\inf_{\alpha \geq 0} \{L(x,-\alpha) - \alpha p\} + L(x,0).
	\ee
	The optimal controls are given in feedback form by 
	\be \label{alpha:feedback}
	\alpha^n_+(t,x) = -\partial_p H_{\uparrow}(x,\Delta_+^n u^n(x)), \qquad 
	\alpha^n_-(t,x) = \partial_p H_{\downarrow}(x,-\Delta_-^n u^n(x)) , \qquad x\in S^n.
	\ee
	Let us remark that $H_{\uparrow}$ and $H_{\downarrow}$ are such that for $(x,p)\in \mathbb{T}\times \R$
	\be
	\begin{split}
		H_{\uparrow}(x,p) + H_{\downarrow}(x,p) = H(x,p)\\
		-\partial_pH_{\uparrow}(x,p) - \partial_pH_{\downarrow}(x,p) = -\partial_pH(x,p). 
	\end{split}
	\ee
	Moreover, $L$ is smooth and uniformly convex in $a$ on compact sets, but $H_\uparrow$ and $H_\downarrow$ are neither uniformly convex nor $\mathcal{C}^2$; however, they are $\mathcal{C}^1$ and $H_\uparrow$, $H_\downarrow$,  $\partial_p H_\uparrow$, $\partial_p H_\downarrow$ are still locally Lipschitz in $(x,p)$.

	The generator of the dynamics of $X^n$ associated to the optimal controls is given by 
	\be
	\label{gen:n}
	\begin{split} 
		\mathcal{L}^n \phi(x) &= \left(\frac{-\partial_p H_{\uparrow}(x,\Delta_+^n u^n(x))}{\dn} +\frac{\sigma}{\dn^2} \right) \left[ \phi(x+\dn) - \phi(x)\right] \\
		&\qquad + \left(\frac{\partial_p H_{\downarrow}(x,-\Delta_-^n u^n(x)) }{\dn} +\frac{\sigma}{\dn^2} \right) \left[ \phi(x-\dn) - \phi(x)\right] \\
		&=-\partial_p H_{\uparrow}(x,\Delta_+^n u^n(x))\Delta_+^n \phi(x) 
		+ \partial_p H_{\downarrow}(x,-\Delta_-^n u^n(x))\Delta_-^n \phi(x) + \sigma\Delta_2^n \phi(x).
	\end{split}
	\ee
	Hence the discrete Fokker-Planck equation associated to this generator is given by
	\be\label{FP:n}
	\begin{split}
		\frac{d}{dt}\mu^n - \sigma \Delta^n_2\mu^n + \Delta_-^n(\partial_pH_{\uparrow}(x,\Delta_+^nu^n(x))\mu^n)\\
		- \Delta_+^n(\partial_pH_{\downarrow}(x,-\Delta_-^nu^n(x))\mu^n) = 0, \qquad x \in S^n.
	\end{split}
	\ee
	
	For a function $U$ defined on  $\mathcal{P}(S^n)$ we denote by 
	$\partial_{m_j} U$ its derivative along the direction $e_j$; and denote equivalently $e_j = e_{x_j}$ and $\partial_{m_j}U= \partial_{m_{x_j}}U$, 
	because we view $m\in\mathcal{P}(S^n)$ as $m= \sum_{j=1}^n m_j \delta_{x_j}$. 
	More precisely, we will consider only derivatives along directions $(e_j-e_i)$, which are tangent vectors to the simplex. 
	The discrete master equation for $U^n:[0,T] \times S^n \times \mathcal{P}(S^n)$ is then given by 
	\be 
	\label{master:n}
	\begin{split}
		&-\partial_t U^n(t,x,m)   +H_{\uparrow}(x,\Delta_+^n U^n(t,x,m)) +H_{\downarrow}(x,-\Delta_-^n U^n(t,x,m)) - \sigma\Delta_2^n U^n(t,x,m) -f(m)(x) \\
		& \qquad -\sum_{y\in S^n} m_y  \left( \frac{-\partial_p H_{\uparrow}(y,\Delta_+^n U^n(y,m))}{\dn} +\frac{\sigma}{\dn^2} \right) \left( \partial_{m_{y+\dn}} U^n(x,m) - \partial_{m_y} U^n(x,m)\right) \\
		& \qquad - \sum_{y\in S^n} m_y  \left( \frac{\partial_p H_{\downarrow}(y,-\Delta_-^n U^n(y,m))}{\dn} +\frac{\sigma}{\dn^2} \right) \left( \partial_{m_{y-\dn}} U^n(x,m) - \partial_{m_y} U^n(x,m)\right) =0
	\end{split}
	\ee 
	It can be derived as for the continuous state space, imposing that $u^n(t,x)= U^n(t,x,\mu^n)$ and applying the chain rule. {Here, by a classical solution to the master equation \eqref{master:n}, we mean a function in $\mathcal{C}^1([0,T] \times \mathcal{P}(S^n); \R^n)$ solving the equation.}
	We note, for future use, that the last two terms are equal to 
	\begin{align}
		&-\int_{\T} m(dy)  \frac{-\partial_p H_{\uparrow}(y,\Delta_+^n U^n(y,m))}{\dn}  \left( \partial_{m_{y+\dn}} U^n(x,m) - \partial_{m_y} U^n(x,m)\right) \label{14}\\
		& - \int_{\T} m(dy)  \frac{\partial_p H_{\downarrow}(y,-\Delta_-^n U^n(y,m))}{\dn}   \left( \partial_{m_{y-\dn}} U^n(x,m) - \partial_{m_y} U^n(x,m)\right) \label{15}\\
		& - \int_{\T} m(dy) \frac{\sigma}{\dn^2} \left( \partial_{m_{y+\dn}} U^n(x,m) - 2 \partial_{m_y} U^n(x,m) + \partial_{m_{y-\dn}} U^n(x,m) \right) \label{16}
	\end{align} 
	recalling  that $m= \sum_{j=1}^n m_j \delta_{x_j}$. 
	
	
	\begin{rem}
		\label{rem:quadratic}
		As an example, consider the simple case of quadratic Lagrangian: $L(a)=\frac{a^2}{2}$. In this case, we have $H_\uparrow(p) = \tfrac12 p_-^2$, $H_\downarrow(p) = \tfrac12 p_+^2$, and thus $\alpha^n_+ = (\Delta^n_+ u)_-$, $\alpha^n_- = (-\Delta^n_- u)_+$, where $p_+$ and $p_-$ denote the positive and negative part of $p$.  
	\end{rem}

	\subsection{Heuristic derivation of the limit master equation} 
	\label{sec:1.2}
	In this section, we give a formal justification of the previous discretization. Ultimately, we want to show that 
	\be 
	\lim_{n\to \infty} U^n(t,x^n,m^n) = U(t,x,m),
	\ee
	when $|x^n -x| + W_1(m^n,m)\to 0$.
	Let us assume that the above holds true and show that, formally, the master equation \eqref{master:n} converges indeed to \eqref{master}. 
	We first have 
	\be 
	\lim_n \Delta_{\pm}^n U^n(t,x,m) = \pm \partial_x U(t,x,m), \qquad 
	\lim_n \Delta_2^n U^n(t,x,m) = \partial_{xx} U(t,x,m)
	\ee
	and thus  the first terms in \eqref{master:n} converge to the corresponding one in \eqref{master}. 
	
	Formally, we should get $U^n(t,x,m) \approx U(t,x, \sum_{j} m_j \delta_{x_j})$:  applying  definition \eqref{delta_m_der} of the measure derivative, 
	we  have 
	\be
	\begin{split} 
		\partial_{m_{y\pm\dn}} U^n(x,m) - \partial_{m_y} U^n(x,m)
		& = \int_\T \frac{\delta U}{\delta m} (x, m; z) (\delta_{y\pm\dn} -\delta_y)(dz) \\
		& = \frac{\delta U}{\delta m} (x,m; y\pm \dn) - \frac{\delta U}{\delta m} (x,m; y) 
	\end{split}
	\ee 
	and hence
	\be 
	\lim_n \frac{\partial_{m_{y\pm\dn}} U^n(x,m) - \partial_{m_y} U^n(x,m)}{\dn} = \pm \partial_y \frac{\delta U}{\delta m}(x,m;y).
	\ee
	Therefore the terms in \eqref{14}-\eqref{15} give  
	\be 
	\begin{split}
		\approx &- \int_\T m(dy) \left[ -\partial_p H_{\uparrow}(y,\partial_x U^n(y,m)) \partial_y \frac{\delta U}{\delta m}(x,m;y) - \partial_p H_{\downarrow}(y,\partial_x U^n(y,m)) \partial_y \frac{\delta U}{\delta m}(x,m;y)\right]  \\
		&= \int_\T m(dy) \partial_pH(y,\partial_x U(y,m)) \partial_y \frac{\delta U}{\delta m}(x,m;y) =- \left \langle \partial_x(\partial_pH(\cdot,\partial_x U(\cdot,m)) m),\frac{\delta U}{\delta m}(x,m;\cdot) \right \rangle.
	\end{split}
	\ee 
	while the term in \eqref{16} yields 
	\be 
	\begin{split}
		&-\sigma \int_\T m(dy) \frac{ \frac{\delta U}{\delta m} (x,m; y+ \dn) 
			- 2\frac{\delta U}{\delta m}(x,m;y) +\frac{\delta U}{\delta m} (x,m; y- \dn)}{ \dn^n } \\
		& \rightarrow - \sigma \int_\T m(dy) \partial^2_y \frac{\delta U}{\delta m}(x,m;y) = -\sigma \left\langle \partial_{xx}m,\frac{\delta U}{\delta m}(x,m;\cdot) \right \rangle.
	\end{split}
	\ee
	This provides the remaining terms in \eqref{master}. 
	
	As far as convergence of the trajectories is concerned, we study it by means of the generators. By \eqref{gen:n} we have
	\begin{align*}
		\mathcal{L}^n \phi(x) &\approx -\partial_pH_{\uparrow}(x,\partial_x U(x,\mu_t)) \partial_x \phi(x) - \partial_p H(x,\partial_x U(x,\mu_t))\partial_x \phi(x) +\sigma \partial_{xx} \phi(x) \\
		&= -\partial_pH(x,\partial_x U(x,\mu_t)) \partial_x \phi(x) 
		+ \sigma\partial_{xx} \phi(x), 
	\end{align*} 
	which is the generator of the limiting dynamics, yielding the convergence in distribution of the optimal processes.

	\section{Convergence results in the absence of a common noise}
	\label{sec:3}

	We prove convergence, with a convergence rate, of the discrete mater equation \eqref{master:n} to \eqref{master} and then of the related optimal trajectories, first in case  \eqref{master} admits a smooth solution and then in case there is no such solution. 
	
	The monotonicity \eqref{mono} of $f$ and $g$ implies that both \eqref{master} and \eqref{master:n} admit at most one classical solution. It also implies the uniqueness of solutions of the systems of characteristics such as \eqref{MFGsystem} or the discrete system
	\be\label{mfgsystem:n}
	\begin{split}
		&-\frac{d}{dt} u^n +H_{\uparrow}(x,\Delta_+^n u^n(x)) +H_{\downarrow}(x,-\Delta_-^n u^n(x)) - \sigma\Delta_2^n u^n(x) =f(\mu^n_t)(x) , \qquad t \in (s,T), x\in S^n,\\
		&\frac{d}{dt}\mu^n  - \sigma \Delta^n_2\mu^n + \Delta_-^n(\partial_pH_{\uparrow}(x,\Delta_+^nu^n(x))\mu^n)\\
		&\qquad - \Delta_+^n(\partial_pH_{\downarrow}(x,-\Delta_-^nu^n(x))\mu^n) = 0, \qquad t \in (s,T), x \in S^n,\\
		& u^n(T,x) = g(\mu^n_T)(x) ; \qquad \mu^n(s) = m^n,
	\end{split}
	\ee
	We first show a uniform bound which is used throughout  the paper. 
	
	\begin{lem}\label{lemma:lip}
		Let $(u^n,\mu^n)$ be a solution of \eqref{mfgsystem:n} for given initial time $t_0 \in [0,T]$.   This solution satisfies
		\be
		\label{bound:U}
		\sup_{x \in S^n} |u^n(s,x)| \leq (T-s)( \|f\|_\infty + \sup_x|\inf_{\alpha}L(x,\alpha)|)+ \|g\|_\infty .
		\ee 
		and then there exists $M > 0$ such that for any $s \in [0,T], \tilde{\mu^n} \in \mathcal{P}(S^n), n \geq 1$ and $x \in S^n$
		\be 
		\label{Lipx:Un}
		|\Delta_{\pm}^n u^n(s,x)|  \leq M.
		\ee 
	\end{lem}
	
	\begin{proof}
		Recall that $u^n$ is the value function of an optimal control problem for the dynamics \eqref{rate:n} and cost \eqref{cost:n} (given $\mu^n$). The lower bound in \eqref{bound:U} is then straightforward, while the upper bound follows by choosing the feedback controls $\alpha_+(x) $ which minimizes the function $a \mapsto L(x,a)$ for $a\geq 0$, and  $- \alpha_-(x) $ which minimizes $a \mapsto L(x,a)$ for $a\leq 0$.  
		In order to prove \eqref{Lipx:Un}, 
		it is convenient to use stochastic open-loop controls for the control problem and thus to consider a probabilistic representation of the dynamics of $X^n$. We now introduce the representation analyzed in \cite{cecchinfischer}.

		Let  $\mathcal{N}$ be a Poisson random measure on $[0,T] \times [0,\infty)^2$ with intensity measure $\nu(d\theta)$ on $[0,\infty)^2$ given by 
		\[
		\nu(E) = \ell( E \cap ( [0,\infty) \times \{0\} ) )
		+ \ell( E \cap (\{0\} \times [0,\infty)) ),
		\]
		where $\ell$ is the Lebesgue measure on $\R$. The measure $\nu$ is in fact the sum of the intersection with the axes and has the property that 
		\[
		\int_{[0,\infty)^2} \phi( \theta_+, \theta_-) \nu (d\theta) = 
		\int_0^\infty \phi(\theta_+, 0) d \theta_+ 
		+ \int_0^\infty \phi(0, \theta_-) d\theta_- .
		\]
		Consider then the dynamics 
		\be 
		\label{dyn:Xn}
		dX^n_t = \int_{[0,\infty)^2} \bigg(\dn \mathbbm{1}_{\left(0, \frac{\sigma}{\dn^2} + \frac{\alpha^n_+(t,X^n_t)}{2 \dn} \right]} (\theta_+) 
		- \dn \mathbbm{1}_{\left(0, \frac{\sigma}{\dn^2} + \frac{\alpha^n_-(t,X^n_t)}{2 \dn} \right]} (\theta_-)  \bigg) \mathcal{N}(d\theta, dt)
		\ee 
		for a control $(\alpha^n_+(t,x),\alpha^n_-(t,x)$. Following \cite{cecchinfischer}, we get that the generator is given by (calling $\lambda(\alpha_+,\alpha_-,\theta)$ the integrand above)
		\[
		\int_{[0,\infty)^2} \left[\phi(x+ \lambda(\alpha_+(t,x),\alpha_-(t,x),\theta)) -\phi(x) \right]  \nu(d\theta)  
		= \Delta^n_+ \phi(x) \alpha_+(t,x) + \Delta^n_-\phi(x)\alpha_-(t,x)+ \sigma \Delta^n_2 \phi(x) ,
		\] 
		which ensures that $X^n$ has the transition rates as in \eqref{rate:n}. The advantage in using the representation \eqref{dyn:Xn} is that it permits to use stochastic open-loop controls (for the strong formulation), which are predictable stochastic processes (with respect to the filtration generated by the fixed Poisson measure and the initial condition). 
		
		Therefore we have (recall that $(\mu^n_t)_{t_0 \leq t\leq T}$ is fixed)
		\[
		u^n(t_0,x) = \inf_{\alpha^n} \E \left[ \int_{t_0}^T L(X^n_t,(\alpha^n_+)_t) + L(X^n_t,-(\alpha^n_-)_t)- L(X^n_t,0) +f( \mu^n_t)(X^n_t) dt + g(\mu^n_T)(X^n_T) \right],
		\]
		where $X^n$ starts at $X^n_{t_0} = x$ and uses the stochastic control $(\alpha^n_+,\alpha^n_-)$.
		As to \eqref{Lipx:Un}, if $(\alpha^n_+,\alpha^n_-)_t$ is optimal for $x$ (that is given by \eqref{alpha:feedback}) and we denote by $X^n$ the process starting at $x$ and by $\tilde{X}^n$ the process starting at $x+\dn$, both with the control $(\alpha^n_+,\alpha^n_-)_t$, then 
		\begin{align*}
			&u(t_0, x+\dn) - u(t_0, x) \\ 
			&\leq \E \left[ \int_{t_0}^T L(\tilde{X}^n_t,(\alpha^n_+)_t) + L(\tilde{X}^n_t,-(\alpha^n_-)_t)- L(\tilde{X}^n_t,0)+f( \mu^n_t)(\tilde{X}^n_t) dt + g(\mu^n_T)(\tilde{X}^n_T) \right]\\
			&\quad - \E \left[ \int_{t_0}^T L(X^n_t,(\alpha^n_+)_t)+ L(X^n_t,-(\alpha^n_-)_t)- L(X^n_t,0) +f( \mu^n_t)(X^n_t) dt + g( \mu^n_T)(X^n_T) \right] \\
			& \leq M \sup_{t_0 \leq t\leq T} \E |\tilde{X}^n_t -X^n_t|,
		\end{align*} 
		where we have used the regularity of $L$ in its first variable. We conclude by noticing that $\tilde{X}^n_t -X^n_t = \dn$ for any $t$, because all the other terms cancel. By changing the roles of $x$ and $x+\dn$, we obtain the opposite inequality and hence \eqref{Lipx:Un} follows. 
	\end{proof}
	Clearly this estimate translates directly to the solution of the master equation thanks to its representation by the characteristics.

	\subsection{Classical solutions} 
	\label{sec:3.1}
	
	We first prove convergence, with a convergence rate, in cases in which \eqref{master} and \eqref{master:n} admit a classical solution. Results on the existence of such solutions are given in \cite[thm. 2.4.2]{cardaliaguet2019master} for the continuous master equation and in \cite{bayraktarcohen, cecchin2019b} for the discrete master equation, assuming regularity of $f$ and $g$ in the measure argument.
	
	The first preliminary result states that a smooth solution of the continuous master equation, computed on discrete measures, almost solves the discrete master equation. 
	
	\begin{prop}
		If $U$ is the classical solution to \eqref{master} then 
		$V^n(t,x, m) := U(t,x, \sum_{j=1}^n m_j \delta_{x_j})$ solves 
		\be 
		\label{Vn}
		\begin{split}
			&-\partial_t V^n(x,m)   +H_{\uparrow}(x,\Delta^n_+ V^n(x,m))  + H_{\downarrow}(x,-\Delta^n_-V^n(x,m))- \sigma\Delta_2^n V^n(x,m) -f(m)(x) \\
			& \qquad + \sum_{y\in S^n} m_y  \frac{\partial_p H_{\uparrow}(u,\Delta^n_+ V^n(y,m))}{\dn} \left( \partial_{m_{y+\dn}}V^n(x,m) - \partial_{m_{y}} V^n(x,m)\right) \\
			& \qquad +\sum_{y\in S^n} m_y  \frac{-\partial_p H_{\downarrow}(u,-\Delta^n_- V^n(y,m))}{\dn} \left( \partial_{m_{y-\dn}}V^n(x,m) - \partial_{m_{y}} V^n(x,m)\right)\\
			& \qquad - \sum_{y\in S^n} m_y \frac{\sigma}{\dn^2}  \left( \partial_{m_{y+\dn}} V^n(x,m) - 2\partial_{m_y} V^n(x,m) +\partial_{m_{y-\dn}} V^n(x,m)\right) = r^n(t,x,m),
		\end{split}
		\ee
		with $|r^n(t,x,m)| \leq C \omega(\frac1n)$, where $\omega$ is a modulus of continuity of $\partial_x U$, $\partial_{xx} U$, $\frac{\delta U}{\delta m}$, $\partial_y  \frac{\delta U}{\delta m}$ and $\partial_{yy}  \frac{\delta U}{\delta m}$.
	\end{prop} 
	
	\begin{proof}
		The proof of this statement follows from the fact that assuming this regularity on $U$, all the heuristics of section \ref{sec:1.2} are true with a modulus of convergence driven by $\omega$.
	\end{proof}

	\begin{thm}
		\label{thm:conv:Un}
		Let $U$ be a classical solution to \eqref{master} and $U^n$ be a classical solution to \eqref{master:n}. 
		There exists a constant $C$ (independent of $n$) such that, for $V^n$ defined as in the previous result,
		\begin{align}
			&|U^n(t,x,m) - V^n(t,x,m)| \leq C \omega\left(\frac1n \right), \qquad 
			\forall t\in [0,T],  x\in S^n, m\in \mathcal{P}(S_n) \label{conv:Un}\\
			&\E \int_0^T | \Delta_+^n (U^n-V^n) (t, X^n_t , \Law(X^n_t))|^2 
			+| \Delta_-^n (U^n-V^n) (t, X^n_t , \Law(X^n_t))|^2  dt 
			\leq C \omega^2\left(\frac1n \right) ,
			\label{conv:Deltan}
		\end{align}
		where $X^n_t$ is the optimal process of the MFG \eqref{rate:n}-\eqref{cost:n}.
	\end{thm}
	
	Notably, in the proof below, we make no use of the monotonicity assumption, as we just use the fact that the solutions to the master equations are classical. We also do not need uniform in $n$ estimates on $U^n$, except for the one of Lemma \ref{lemma:lip}, but we make use of the non-degeneracy of the diffusion. The proof is inspired by the argument of \cite[Thm. 2.2.1]{cardaliaguet2019master} for the convergence of the $N$-player game, which is itself inspired by the stability argument for forward-backward SDEs. 
	Clearly, the best convergence rate in \eqref{conv:Un} is given by $\frac1n$ in case the derivatives of $U$ involved are Lipschitz-continuous.  
	
	\begin{proof}
		Consider any initial time $t_0\in [0,T)$ and distribution $\mu_0 = (\mu_{0,y})_{y\in S^n} \in \mathcal{P}(S^n)$ such that $\mu_{0,y}\neq 0$ for each $y\in S^n$, i.e. $\mu_0$ belongs to the interior of the simplex, and let $(X^n_t)_{t_0\leq t\leq T}$ the optimal process of the discrete MFG starting at $(t_0, \mu_0)$. Denote its law by $\big(\mu^n_t =\Law(X^n_t) \big)_{t_0\leq t\leq T}$. 
		We denote $Z^n=U^n-V^n$, $Z^n_t =U^n_t-V^n_t =(U^n-V^n) (t,X^n_t,\mu^n_t)$, and 
		$r^n_t=r^n(t,X^n_t,\mu^n_t)$.
		For any $t\in [t_0,T]$, applying It\^o formula  to $|U^n-V^n|^2(t,X^n_t,\mu^n_t)$ and then taking conditional expectation with respect to the initial condition (denoted $\E_0$) and  using  equations \eqref{master:n} and \eqref{Vn} and standard algebraic manipulations, we obtain
		\begin{align*}
			&\E_0 |Z^n_T|^2 - \E_0 |Z^n_t|^2  
			=\E_0 \int_t^T \bigg[ \left(|Z^n (X^n_s + \dn, \mu^n_s)|^2 - |Z_s^n|^2 \right) \left( \frac{\sigma}{\dn^2} -\frac{\partial_pH_{\uparrow}(X^n_s,\Delta^n_+ U^n_s)}{\dn} \right) \\
			&\quad + \left( |Z^n (X^n_s - \dn,\mu^n_s)|^2 - |Z_s^n|^2 \right) \left( \frac{\sigma}{\dn^2} 
			+\frac{\partial_pH_{\downarrow}(X^n_s,-\Delta^n_- U^n_s)}{\dn} \right) \\
			& \quad + 2 Z_s^n  ( \partial_t U^n_s - \partial_t V^n_s ) \\
			& \quad + 2 Z_s^n ( D^m U^n - D^m V^n)(X^n_s, \mu^n_s) \cdot \frac{d}{dt}\mu_s^n \bigg] ds\\
			&= \E_0 \int_t^T \bigg[ \left( |Z^n  (X^n_s + \dn, \mu^n_s) - Z^n_s|^2 
			+ 2 Z^n_s ( Z^n  (X^n_s + \dn, \mu^n_s) - Z^n_s ) \right) \left( \frac{\sigma}{\dn^2} -\frac{\partial_pH_{\uparrow}(X^n_s,\Delta^n_+ U^n_s)}{\dn}  \right) \\
			&\quad + \left( |Z^n  (X^n_s - \dn, \mu^n_s) - Z^n_s|^2 + 2 Z^n_s ( Z^n  (X^n_s - \dn, \mu^n_s) - Z^n_s ) \right) \left(\frac{\sigma}{\dn^2} +\frac{\partial_pH_{\downarrow}(X^n_s,-\Delta^n_- U^n_s)}{\dn} \right) \\
			&\quad + 2 Z^n_s 
			\bigg(
			H_{\uparrow}(X^n_s,\Delta^n_+ U^n_s) + H_{\downarrow}(X^n_s,-\Delta^n_-U^n_s)
			- H_{\uparrow}(X^n_s,\Delta^n_+ V^n_s) - H_{\downarrow}(X^n_s,-\Delta^n_- V^n_s)
			- \sigma\Delta_2^n U^n_s +\sigma\Delta_2^n V^n_s  \\
			& \qquad + \sum_{y\in S^n} \mu^n_{s,y}  \frac{\partial_pH_{\uparrow}(y,\Delta^n_+ U^n(y,\mu^n_s))}{\dn} \left( \partial_{m_{y+\dn}}U^n_s - \partial_{m_{y}} U^n_s\right) \\
			&\qquad \qquad - \sum_{y\in S^n} \mu^n_{s,y}  \frac{\partial_pH_{\uparrow}(y,\Delta^n_+ V^n(y,\mu^n_s))}{\dn} \left( \partial_{m_{y+\dn}}V^n_s - \partial_{m_{y}} V^n_s\right)
			\\
			& \qquad - \sum_{y\in S^n} \mu^n_{s,y}  \frac{\partial_pH_{\downarrow}(y,-\Delta^n_- U^n(y,\mu^n_s))}{\dn} \left( \partial_{m_{y-\dn}}U^n_s - \partial_{m_{y}} U^n_s\right) \\
			&\qquad \qquad + \sum_{y\in S^n} \mu^n_{s,y}  \frac{\partial_pH_{\downarrow}(y,-\Delta^n_- V^n(y,\mu^n_s))}{\dn} \left( \partial_{m_{y-\dn}}V^n_s - \partial_{m_{y}} V^n_s\right)
			\\
			& \qquad - \sum_{y\in S^n} \mu^n_{s,y} \frac{\sigma}{\dn^2}  \left( \partial_{m_{y+\dn}} U^n_s - 2\partial_{m_y} U^n_s +\partial_{m_{y-\dn}} U^n_s\right) \\
			&  \qquad \qquad + \sum_{y\in S^n} \mu^n_{s,y} \frac{\sigma}{\dn^2}  \left( \partial_{m_{y+\dn}} V^n_s - 2\partial_{m_y} V^n_s +\partial_{m_{y-\dn}} V^n_s\right)
			\bigg) + r^n_s \\
			&\quad + 2 Z^n_s 
			\bigg( - \sum_{y\in S^n} \mu^n_{s,y}  \frac{\partial_pH_{\uparrow}(y,\Delta^n_+ U^n(y,\mu^n_s))}{\dn} \left( \partial_{m_{y+\dn}}U^n_s - \partial_{m_{y}} U^n_s\right) \\
			&\qquad \qquad +\sum_{y\in S^n} \mu^n_{s,y}  \frac{\partial_pH_{\downarrow}(y,-\Delta^n_- U^n(y,\mu^n_s))}{\dn} \left( \partial_{m_{y-\dn}}U^n_s - \partial_{m_{y}} U^n_s\right) \\
			&\qquad \qquad 
			+ \sum_{y\in S^n} \mu^n_{s,y}  \frac{\partial_pH_{\uparrow}(y,\Delta^n_+U^n(y,\mu^n_s))}{\dn} \left( \partial_{m_{y+\dn}}V^n_s - \partial_{m_{y}} V^n_s\right)
			\\
			&\qquad \qquad 
			- \sum_{y\in S^n} \mu^n_{s,y}  \frac{\partial_pH_{\downarrow}(y,-\Delta^n_-U^n(y,\mu^n_s))}{\dn} \left( \partial_{m_{y-\dn}}V^n_s - \partial_{m_{y}} V^n_s\right)
			\\
			& \qquad + \sum_{y\in S^n} \mu^n_{s,y} \frac{\sigma}{\dn^2}  \left( \partial_{m_{y+\dn}} U^n_s - 2\partial_{m_y} U^n_s +\partial_{m_{y-\dn}} U^n_s\right) \\
			&  \qquad \qquad - \sum_{y\in S^n} \mu^n_{s,y} \frac{\sigma}{\dn^2}  \left( \partial_{m_{y+\dn}} V^n_s - 2\partial_{m_y} V^n_s +\partial_{m_{y-\dn}} V^n_s\right)
			\bigg) \bigg] ds \\
			&=  \E_0 \int_t^T \bigg[ 
			|\Delta_+^n Z^n_s|^2 (\sigma- \dn \partial_pH_{\uparrow}(X^n_s,\Delta^n_+ U^n_s)) + |\Delta_-^n Z^n_s|^2 (\sigma+ \dn \partial_pH_{\downarrow}(X^n_s,-\Delta^n_- U^n_s)) \\
			&\quad +2 Z^n_s ( \sigma\Delta^n_2 Z^n_s - \partial_pH_{\uparrow}(X^n_s,\Delta^n_+ U^n_s) \Delta^n_+ Z^n_s + \partial_pH_{\downarrow}(X^n_s,-\Delta^n_-U^n_s)\Delta^n_- Z^n_s) \\
			& \quad +2W^n_s \bigg(H_{\uparrow}(X^n_s,\Delta^n_+ U^n_s) + H_{\downarrow}(X^n_s,-\Delta^n_-U^n_s)
			- H_{\uparrow}(X^n_s,\Delta^n_+ V^n_s) - H_{\downarrow}(X^n_s,-\Delta^n_- V^n_s)
			- \sigma\Delta_2^n Z^n_s +r^n_s \\
			&\qquad + \sum_{y\in S^n} \mu^n_{s,y}  (\partial_pH_{\uparrow}(y,\Delta^n_+ U^n(y,\mu^n_s))
			-\partial_pH_{\uparrow}(y,\Delta^n_+ V^n(y,\mu^n_s))) 
			\frac{  \partial_{m_{y+\dn}}V^n_s - \partial_{m_{y}} V^n_s }{ \dn}\\
			&\qquad - \sum_{y\in S^n} \mu^n_{s,y}  (\partial_pH_{\downarrow}(y,-\Delta^n_- U^n(y,\mu^n_s))
			-\partial_pH_{\downarrow}(y,-\Delta^n_- V^n(y,\mu^n_s))) 
			\frac{  \partial_{m_{y-\dn}}V^n_s - \partial_{m_{y}} V^n_s } {\dn} \bigg) \bigg]ds \\
			&=  \E_0 \int_t^T \bigg[ 
			|\Delta_+^n Z^n_s|^2 (\sigma- \dn \partial_pH_{\uparrow}(X^n_s,\Delta^n_+ U^n_s)) + |\Delta_-^n Z^n_s|^2 (\sigma+ \dn \partial_pH_{\downarrow}(X^n_s,-\Delta^n_- U^n_s)) \\
			& \quad +2W^n_s \bigg(- \partial_pH_{\uparrow}(X^n_s,\Delta^n_+ U^n_s) \Delta^n_+ Z^n_s +H_{\uparrow}(X^n_s,\Delta^n_+ U^n_s)
			- H_{\uparrow}(X^n_s,\Delta^n_+ V^n_s)\\
			& \qquad \qquad+ \partial_pH_{\downarrow}(X^n_s,-\Delta^n_- U^n_s) \Delta^n_- Z^n_s +H_{\downarrow}(X^n_s,-\Delta^n_- U^n_s)
			- H_{\downarrow}(X^n_s,-\Delta^n_- V^n_s)+r^n_s \\
			&\qquad + \sum_{y\in S^n} \mu^n_{s,y}  (\partial_pH_{\uparrow}(y,\Delta^n_+ U^n(y,\mu^n_s))
			-\partial_pH_{\uparrow}(y,\Delta^n_+ V^n(y,\mu^n_s))) 
			\frac{  \partial_{m_{y+\dn}}V^n_s - \partial_{m_{y}} V^n_s }{ \dn}\\
			&\qquad - \sum_{y\in S^n} \mu^n_{s,y}  (\partial_pH_{\downarrow}(y,-\Delta^n_- U^n(y,\mu^n_s))
			-\partial_pH_{\downarrow}(y,-\Delta^n_- V^n(y,\mu^n_s))) 
			\frac{  \partial_{m_{y-\dn}}V^n_s - \partial_{m_{y}} V^n_s } {\dn} \bigg) \bigg]ds.
		\end{align*} 
		We recall that $-\partial_pH_{\uparrow}(X^n_s,\Delta^n_{+} U^n_s)$ and $\partial_pH_{\downarrow }(X^n_s,-\Delta^n_{\mp} U^n_s)$ are the optimal transition rates and thus they are non-negative. 
		Since $|\Delta^n_{\pm} U^n_s|\leq M$
		and 
		$W_T =0$, using the convexity inequality $AB \leq \epsilon A^2 + \frac{1}{4\epsilon} B^2$ and the bounds 
		\[
		\left| \frac{  \partial_{m_{y\pm\dn}}V^n_s - \partial_{m_{y}} V^n_s }{\dn} \mp D^m U( X^n_s, \mu^n_s) \right| \leq \omega \Big(\frac1n\Big)
		\] 
		and $|D^m U| \leq C$, as well as the fact that $H_{\uparrow}$ and $H_{\downarrow}$ and their derivatives are locally Lipschitz,  
		we obtain 
		\begin{align*}
			\E_0 &|Z^n_t|^2 + \sigma \E_0 \int_t^T \left(|\Delta_+^n Z^n_s|^2  + |\Delta_-^n Z^n_s|^2 \right) ds \\
			&\leq C \E_0 \int_t^T |Z^n_s|  \bigg( |\Delta^n_+ Z^n_s| + |\Delta^n_- Z^n_s|  
			+ \sum_{y\in S^n} \mu^n_{s,y}  |\partial_pH_{\uparrow}(y,\Delta^n_+ U^n(y,\mu^n_s))
			-\partial_pH_{\uparrow}(y,\Delta^n V^n_+(y,\mu^n_s))| \\
			& \qquad \qquad \quad +\sum_{y\in S^n} \mu^n_{s,y}  |\partial_pH_{\downarrow}(y,-\Delta^n_- U^n(y,\mu^n_s))
			-\partial_pH_{\downarrow}(y,-\Delta^n_- V^n(y,\mu^n_s))| + |r^n_s| + \omega \Big(\frac1n \Big)\bigg) ds \\
			&\leq C \E_0 \int_t^T |Z^n_s|  \left( |\Delta^n_+ Z^n_s| + |\Delta^n_- Z^n_s|   + \omega \Big(\frac1n \Big)  
			\right) ds \\
			&\leq C\E_0 \int_t^T |Z^n_s|^2 ds + \frac{\sigma}{2} \E_0 \int_t^T \left(|\Delta_+^n Z^n_s|^2  + |\Delta_-^n Z^n_s|^2 \right) ds 
			+C \omega^2 \Big( \frac1n \Big) .
		\end{align*}
		This gives  
		\be 
		\label{eq:38}
		\E_0 |Z^n_t|^2 +\frac \sigma2 \E_0 \int_t^T \left(|\Delta_+^n Z^n_s|^2  + |\Delta_-^n Z^n_s|^2 \right)ds \leq C\E_0 \int_t^T |Z^n_s|^2 ds + C \omega^2 \Big( \frac1n \Big) 
		\ee 
		and thus Gronwall's inequality yields 
		\be 
		\label{eq:39}
		\sup_{t\in [t_0,T]} \E_0 |Z^n_t|^2 \leq C \omega^2 \Big( \frac1n \Big)
		\qquad \P - a.s.  
		\ee 
		At $t=t_0$, the above inequality gives 
		\be 
		|U^n(t_0, X^n_{t_0}, \mu_0) -V^n (t_0, X^n_{t_0}, \mu_0)| \leq C \omega \Big( \frac1n \Big)  \qquad \P - a.s. ,
		\ee  
		which, since $\Law(X_{t_0}) =\mu_0$ is supported on the entire $S^n$, provides 
		\be 
		|U^n(t_0, x, \mu_0) -V^n (t_0, x, \mu_0)| \leq C \omega \Big( \frac1n \Big)
		\ee 
		for any $t_0 \in [0,T]$, $x\in S^n$ and  $\mu_0$ in the interior of $\mathcal{P}(S^n)$. Since $U^n$ and $V^n$ are continuous in the measure argument, the above inequality holds for any $\mu\in \mathcal{P}(S^n)$, which provides \eqref{conv:Un}, but only for $n \geq 4M$; changing the value of the constant, \eqref{conv:Un} holds for any $n$. 
		
		Finally, letting $t_0 = 0$, applying \eqref{eq:39} into \eqref{eq:38} and taking the expectation, we obtain \eqref{conv:Deltan}. 
	\end{proof}
	
	We now turn to the convergence of the trajectories at equilibrium. Consider an initial distribution (at time 0) $m_0$ of the limit MFG, and a random variable $\xi$ (with values in $\T$)  with $\Law(\xi)=m_0$. For the discretization, let $E^n_i= [x^n_i-\frac{1}{2n}, x^n_i-\frac{1}{2n})$ and 
	\be 
	m^n_0 = \sum_{i=1}^n m_0(E^n_i) \delta_{x^n_i}  , \qquad
	\xi^n = \sum_{i=1}^n x^n_i \mathbbm{1}_{\{\xi\in E^n_i\}}.
	\ee
	We have $\Law(\xi^n)=m^n_0$ and $\sqrt[k]{\E |\xi_n-\xi|^k}\leq \frac{1}{2n}$ for any integer $k\geq 1$. 
	
	Let $X^n$ be the trajectory at equilibrium for the discrete MFG with initial condition $\xi^n$. Hence the control of the players is given by $\alpha^n_+(t,x) = - \partial_pH_{\uparrow}(x,\Delta^n_+ U^n(t,x, \Law(X^n_t)))$ (and similarly for $\alpha^n_-$), where $U^n$ is the classical solution to \eqref{master:n}. Let also $X$ be the optimal process for the limit MFG \eqref{MFGsystem} with initial condition $\xi$. The associated control is thus given by $\alpha(t,x) = -\partial_pH(x,\partial_x U(t,x,\Law(X_t)))$, where $U$ is the classical solution of \eqref{master}. 

	\begin{thm}[Convergence of trajectories] 
		\label{thm:conv:traj}
		Under the assumptions of Theorem \ref{thm:conv:Un}, using the above notations, we have
		\be 
		\label{conv:47}
		\sup_{0\leq t\leq T} W_1( \Law(X^n_t), \Law(X_t)) \leq  
		C \omega\left(\frac1n\right) +   \frac{C}{n^{\frac13}}
		\ee
		and further 
		\be 
		\label{conv:traj}
		\lim_n X^n = X    \qquad \mbox{ in law in } \mathcal{D}([0,T], \T).
		\ee
		
	\end{thm} 
	
	\begin{proof} 
		Let $\bar{X}^n$ and $\tilde{X}^n$ be the processes starting at $\xi^n$, with dynamics given by \eqref{rate:n}, with  controls therein given  by $(\bar{\alpha}^n_+(t,x),\bar{\alpha}^n_-(t,x)) = (- \partial_pH_{\uparrow}(x,\Delta^n_+ U (t, x, \Law(\bar{X}^n_t) )), \partial_pH_{\downarrow}(x,-\Delta^n_-U(t,x,\Law(\bar{X}^n_t))))$ and 
		$(\tilde{\alpha}^n_+(t,x),\tilde{\alpha}^n_-(t,x)) = (- \partial_pH_{\uparrow}(x,\partial_x U (t, x, \Law(X_t) )),\partial_pH_{\downarrow}(x,\partial_x U(t,x,\Law(X_t))))$ respectively.
		The SDE representation \eqref{dyn:Xn}, applying then \eqref{conv:Deltan}, Jensen's inequality and the Lipschitz continuity of $\partial_x U$ in $x$ and $m$ (in $W_1$) and recalling that $W_1(\Law(X),\Law(Y)) \leq \E|X-Y|$, give
		\begin{align*}
			\E \sup_{0\leq s\leq t} | X^n_s - \bar{X}^n_s|
			&\leq \E \int_0^t |\partial_pH_{\uparrow}(X^n_s,\Delta^n_+ U^n(s,X^n_s, \Law(X^n_s))) - \partial_pH_{\uparrow}(\bar X^n_s,\Delta^n_+ U(s,\bar{X}^n_s, \Law(\bar{X}^n_s))) |\\
			&\qquad + |\partial_pH_{\downarrow}(X^n_s,-\Delta^n_- U^n(s,X^n_s, \Law(X^n_s))) - \partial_pH_{\downarrow}(\bar X^n_s,-\Delta^n_- U(s,\bar{X}^n_s, \Law(\bar{X}^n_s))) |ds \\
			& \leq  \E \int_0^t |\partial_pH_{\uparrow}(X^n_s,\Delta^n_+ U^n(s,X^n_s, \Law(X^n_s))) - \partial_pH_{\uparrow}(X^n_s,\Delta^n_+ U(s,X^n_s, \Law(X^n_s))) |\\
			&\qquad +  |\partial_pH_{\downarrow}(X^n_s,-\Delta^n_- U^n(s,X^n_s, \Law(X^n_s))) - \partial_pH_{\downarrow}(X^n_s,-\Delta^n_- U(s,X^n_s, \Law(X^n_s))) |ds \\
			&\qquad + \E \int_0^t |\partial_pH_{\uparrow}(X^n_s,\Delta^n_+ U(s,X^n_s, \Law(X^n_s))) - \partial_pH_{\uparrow}(\bar X^n_s,\Delta^n_+ U(s,\bar{X}^n_s, \Law(\bar{X}^n_s))) |\\
			& \qquad + |\partial_pH_{\downarrow}(X^n_s,-\Delta^n_- U(s,X^n_s, \Law(X^n_s))) - \partial_pH_{\downarrow}(\bar X^n_s,-\Delta^n_- U(s,\bar{X}^n_s, \Law(\bar{X}^n_s))) |ds \\
			&\leq C \omega\left(\frac1n\right) 
			+ C \E \int_0^t \sup_{0\leq r\leq s} | X^n_r - \bar{X}^n_r|ds,
		\end{align*}
		and therefore Gronwall's lemma yields
		\be 
		\label{eq:47}
		\E \sup_{0\leq t\leq T} | X^n_t - \bar{X}^n_t| \leq C \omega\left(\frac1n\right).
		\ee  
		Similarly, recalling that $||\Delta^n_{\pm} U \mp \partial_xU ||_\infty\leq \omega(\frac1n)$, we have 
		\begin{align*}
			\E &\sup_{0\leq s\leq t} | \tilde{X}^n_s - \bar{X}^n_s| 
			\leq \E \int_0^t |\partial_pH_{\uparrow}(\tilde X^n_s,\partial_x U(s,\tilde{X}^n_s, \Law(X_s))) - \partial_p H_{\uparrow}(\bar X^n_s,\Delta^n_+ U(s,\bar{X}^n_s, \Law(\bar{X}^n_s)) )|\\
			& \qquad \qquad \qquad \qquad \qquad |\partial_pH_{\downarrow}(\tilde X^n_s,\partial_x U(s,\tilde{X}^n_s, \Law(X_s))) - \partial_p H_{\downarrow}(\bar X^n_s,-\Delta^n_- U(s,\bar{X}^n_s, \Law(\bar{X}^n_s)) )| ds \\
			&\leq C \omega\left(\frac1n\right) 
			+ C \E \int_0^t | \tilde{X}^n_s - \bar{X}^n_s| + W_1( \Law( \tilde{X}^n_s), \Law( \bar{X}^n_s) ) + W_1( \Law( \tilde{X}^n_s), \Law( X_s)) ds \\
			&\leq C\omega\left(\frac1n\right) 
			+ C\sup_{0\leq t\leq T} W_1( \Law( \tilde{X}^n_t), \Law( X_t)) 
			+C \E \int_0^t \sup_{0\leq r\leq s} | \tilde{X}^n_r - \bar{X}^n_r| ds
		\end{align*} 
		and thus, applying Gronwall's inequality, we get
		\[
		\E \sup_{0\leq t\leq T} | \tilde{X}^n_t - \bar{X}^n_t|  \leq C\omega\left(\frac1n\right) 
		+ C\sup_{0\leq t\leq T} W_1( \Law( \tilde{X}^n_t), \Law( X_t)),  
		\] 
		which, together with \eqref{eq:47}, implies
		\be 
		\label{eq:48}
		\E \sup_{0\leq t\leq T} | \tilde{X}^n_t - X^n_t|  \leq C\omega\left(\frac1n\right) 
		+ C\sup_{0\leq t\leq T} W_1( \Law( \tilde{X}^n_t), \Law( X_t)).
		\ee 
		By Proposition \ref{lem:7} in the appendix, we have convergence in law of $\tilde{X}^n$ t $X$ and the estimate  
		\be 
		\label{rate:W2}
		\sup_{0\leq t\leq T} W_1( \Law( \tilde{X}^n_t), \Law( X_t)) \leq \frac{C}{n^{\frac13}}, 
		\ee
		which, applied in \eqref{eq:48}, yield the claims. 
		%
		%
	\end{proof}

	\subsubsection{Another discretization} 
	
	We recall that, if the cost coefficients are monotone and sufficiently regular in the measure argument, then there exist a solution to the continuous master equation \eqref{master}; see \cite[Thm. 2.4.2]{cardaliaguet2019master}.  As a matter of fact, in that result, the Hamiltonian $H$ is also required to be smooth. Above, we assume that the discrete master equation \eqref{master:n} also possesses a classical solution, in order to apply the chain rule. Such solution is shown to exists in the literarure, assuming that the discrete Hamiltonian is smooth, in particular $\mathcal{C}^2$ in the adjoint variable; see \cite{bayraktarcohen, cecchin2019b} and also \cite{delaruecetraro}. However, the Hamiltonians we consider in \eqref{Hn} is not $\mathcal{C}^2$ in general --see remark \ref{rem:quadratic}-- and thus the existence of a classical solution is not entirely clear in this case. 
	
	Hence we provide here, for completeness, another discretization of the MFG for which the discrete hamiltonian is $\mathcal{C}^2$. Such discretization is possible only in presence of the Laplacian and thus it is not the usual one considered in numerical analysis. For this reason we consider in the paper only the discretization \eqref{rate:n}-\eqref{cost:n}, but  all the argumets could be easily adapted to the following discrete model.    
	Consider a single control 
	$\alpha^n: [0,T]\times S^n \rightarrow \R$, the transition rates 
	\be
	\label{rate:n2}
	\P( X^n_{t+\Delta t} = x^n_{i\pm 1} | X^n_t = x^n_i) = \left(\pm\frac{\alpha^n(t,x^n_i)}{2 \dn} + \frac{\sigma}{\dn^2} \right) \Delta t +o(\Delta t)
	\ee
	and the cost 
	\be 
	\label{cost:n2}
	J^n(\alpha^n, \mu^n) = \E\left[ \int_0^T L(X^n_t,\alpha^n(t,X^n_t)) 
	+ f(\mu^n_t)(X^n_t) dt +g(\mu^n_T)(X^n_T) \right].
	\ee
	Note that the rates are non-negative  if 
	\be 
	|\alpha(t,x)|\leq \frac{2\sigma}{\dn} =2\sigma n ,
	\ee  
	which should a posteriori hold true for $n$ large enough. 
	This formulation provides a smooth Hamiltonian, and is allowed because of the additional viscosity. 
	Denoting 
	\be 
	\Delta^n u(x) = \frac{u(x+\dn) -u(x-\dn)}{2\dn}
	\ee 
	we derive the HJB equation 
	\be 
	\label{HJB:n2}
	- \partial_t u^n -\sigma \Delta^n_2 u^n(x) +H(x,\Delta^n u^n(x)),  =f(\mu^n_t)(x)
	\ee 
	where $H$ is the same Hamiltonian $H$ of the continuous model, 
	and the optimal control is 
	\be
	\alpha^n(t,x) = -\partial_p H(x, \Delta^n u^n(t,x)).
	\ee
	Therefore the discrete master equation becomes 
	\be 
	\label{master:n2}
	\begin{split}
		&-\partial_t U^n(x,m)   +H(x,\Delta^n U^n(x,m)) - \sigma\Delta_2^n U^n(x,m) -f(m)(x) \\
		& \qquad + \sum_{y\in S^n} m_y  \frac{\partial_pH(y,\Delta^n U^n(y,m))}{2\dn} \left( \partial_{m_{y+\dn}}U^n(x,m) - \partial_{m_{y-\dn}} U^n(x,m)\right) \\
		& \qquad - \sum_{y\in S^n} m_y \frac{\sigma}{\dn^2}  \left( \partial_{m_{y+\dn}} U^n(x,m) - 2\partial_{m_y} U^n(x,m) +\partial_{m_{y-\dn}} U^n(x,m)\right) =0.
	\end{split}
	\ee 
	and we can formally see the convergence of the above equation to \eqref{master} as in \S \ref{sec:1.2}.
	

	\subsection{Convergence through the MFG system} 
	\label{sec:3.2} 
	
	Without assuming regularity on $f$ and $g$ in the measure argument, thus without having a classical solutions to the master equation, convergence results can still be established, namely by using the MFG system, which represents the characteristic curves of the master equation. 
	
	\begin{thm}\label{thm:conv5}
		If $\gamma \geq \frac13$ in the standing assumptions {in \S \ref{S:2.2}}, then 
		for any $n$ 
		\be 
		\label{conv:62}
		|U^n(t,x,m) - U(t,x,m)| \leq \frac{C}{n^{\frac16}}, \qquad 
		\forall t\in [0,T],  x\in S^n, m\in \mathcal{P}(S_n).
		\ee
		Moreover, let $X^n$ be the state process of players which plays optimally at the equilibrium in the MFG \eqref{rate:n}-\eqref{cost:n}, with initial distribution $m^n_0$ at $t = 0$, let $X$ be the state process of a player which plays optimally at equilibrium in the limit MFG \eqref{dyn}-\eqref{cost}, with initial condition $m_0$ at $t = 0$. Then $W_1(m^n_0, m_0) \leq \frac{C}{n}$ implies
		\be 
		\label{conv:63}
		\sup_{0\leq t\leq T} W_1( \Law(X^n_t), \Law(X_t)) \leq \frac{C}{n^{\frac16}}
		\ee
		and further 
		\be 
		\label{conv:64}
		\lim_n X^n = X    \qquad \mbox{ in law in } \mathcal{D}([0,T], \T).
		\ee
	\end{thm}
	
	We remark that the result can be equivalently written in terms of the sole MFG systems \eqref{MFGsystem} and \eqref{HJB:n}-\eqref{FP:n}: denoting their unique solutions by $(u,m)$ and $(u^n, m^n)$, we have 
	\be 
	\sup_{0\leq t \leq T} \sup_{x\in S^n} |u^n(t,x) - u(t,x)| 
	+ \sup_{0\leq t \leq T} W_1(m^n_t, m_t)  \leq \frac{C}{n^{\frac16}}.
	\ee 
	The proof is inspired by the arguments of stability of the MFG system under monotonicity. 
	
	Without assuming that $\gamma \geq \frac13$,  we immediately obtain, from the proof below,  the convergence rate  $\min \{ \frac{\gamma}{2} , \frac16 \}$.

	\begin{proof}
		Fix $m_0$ and $m^n_0$ such that $W_1(m^n_0, m_0) \leq \frac{C}{n}$ and an initial time $t_0$.  We consider a solution $(u,m)$ of the MFG system \eqref{MFGsystem} starting at $(t_0,m_0)$, and denote by $X$ the corresponding optimal process given by \eqref{dyn} with $\alpha(t,x) = -\partial_pH(x,\partial_x u(t,x))$ therein.  
		
		\textbf{Step 1}.
		The assumption of the space regularity of $f$ and $g$ gives that $u(t, \cdot)\in \mathcal{C}^{2+\gamma}(\T)$ in space, uniformly in time; 
		see e.g. \cite[Thm. 1.5]{notescetraro} and \cite{gomes_book}. 
		Let $\hat{X}^n$ be the Markov chain given by \eqref{rate:n} with  
		$\alpha_{+}(t,x) = -\partial_pH_{\uparrow}(x,\partial_x u(t,x))$,   $\alpha_-(t,x) = \partial_pH_{\downarrow}(x,\partial_x u(t,x))$, and $\tilde{X}^n$ be given by \eqref{rate:n} with $\tilde{\alpha}_{+}(t,x) = -\partial_pH_{\uparrow}(x,\Delta^n_+ u(t,x))$, $\tilde{\alpha}_-(t,x) = \partial_pH_{\downarrow}(x,-\Delta^n_- u(t,x)) $. Since $\partial_x u$ is Lipschitz in $\T$, we easily derive
		\be 
		\label{eq:54}
		\E \Big[ \sup_{t_0\leq t\leq T} | \tilde{X}^n - \hat{X}^n_t | \Big] \leq \frac Cn .
		\ee 
		Proposition \ref{lem:7} in the appendix gives the 
		convergence in law of $\hat{X}^n$ to $X$, in $\mathcal{D}([0,T], \T)$, 
		and the estimate  
		\be 
		\sup_{0\leq t\leq T} W_1( \Law( \hat{X}^n_t), \Law( X_t)) \leq \frac{C}{n^{\frac13}}. 
		\ee
		Thus \eqref{eq:54} yields 
		%
		%
		\be 
		\label{conv:65}
		\lim_n \tilde{X}^n = X    \qquad \mbox{ in law in } \mathcal{D}([0,T], \T)
		\ee
		and
		\be 
		\label{rate:W1}
		\sup_{0\leq t\leq T} W_1( \Law( \tilde{X}^n_t), \Law( X_t)) \leq \frac{C}{n^\frac13}. 
		\ee
		Since $\sup_{t_0\leq t \leq T} ||u(t, \cdot)||_{2+\gamma} \leq C$, using \eqref{rate:W1}, the Lipschitz-continuity of $f$ and $g$, and \eqref{MFGsystem}, we obtain
		\be 
		\label{HJB:approx}
		\begin{cases} 
			-\partial_t u +H_{\uparrow}(x,\Delta_+^n u(x)) +H_{\downarrow}(x,-\Delta_-^n u(x)) - \sigma\Delta_2^n u(x) =f(x,\tilde{\mu}^n_t) +r^n(t,x) , \quad x\in S^n, & \\
			u(T,x) = g(x, \tilde{\mu}^n_t) + r^n(T,x),  &
		\end{cases}
		\ee
		with $|r^n(t,x)| \leq \frac{C}{n^\frac13}$, if $\gamma \geq \frac13$, whereas $\tilde{\mu}^n_t = \Law(\tilde{X}^n_t)$. Thus $(u,m)$ almost solves the discrete MFG system \eqref{mfgsystem:n}. 
		
		\textbf{Step 2}.
		Let $(u^n,\mu^n)$ be the solutions to the MFG \eqref{mfgsystem:n} starting at $(t_0, m^n_0)$ and $X^n$ the associated state process. Denote the initial random distribution on the states by $\xi^n$. We stress that this is the same initial condition as $\tilde{X}^n$. Thanks to \eqref{HJB:approx}, $u$ (restricted to $S^n$) can be seen as the value function of a control problem with dynamics \eqref{rate:n} and cost
		\be 
		\label{cost:ntilde}
		\begin{split}
			\tilde{J}^n(\alpha, \tilde{\mu}^n) = \E \bigg[ \int_0^T &L(X^n_t,\alpha_+^n(t,X^n_t))+
			L(X^n_t,-\alpha_-^n(t,X^n_t)) - L(X^n_t, 0) + f(X^n_t,\tilde{\mu}^n_t) +r^n(t,X^n_t) dt\\
			& +g(X^n_T, \tilde{\mu}^n_T) +r^n(T, X^n_T) \bigg].
		\end{split}
		\ee
		
		We first compute $u^n$ on $\tilde{X}^n$: denoting 
		$\tilde{\alpha}^n_{+}(t,x)= -\partial_pH_{\uparrow}(x,\Delta^n_{+}u(t,x))$,   $\tilde{\alpha}^n_-(t,x) = \partial_pH_{\downarrow}(x,-\Delta^n_-u(t,x))$
		and $\alpha^n_{+}(t,x)= -\partial_pH_{\uparrow}(x,\Delta^n_{+}u^n(t,x)), \alpha^n_-(t,x) = \partial_pH_{\downarrow}(x,-\Delta^n_-u^n(t,x))$,
		\begin{align*}
			&\E[ u^n(t_0,\xi^n)] = \E[ u^n(T, \tilde{X}^n_T) ] + \E \int_{t_0}^T
			\left( -\partial_t u^n  - \sigma\Delta^n_2 u^n - \tilde{\alpha}^n_+ \Delta^n_+ u^n -\tilde{\alpha}^n_- \Delta^n_- u^n \right) (s,\tilde{X}^n_s) ds \\
			&= \E \bigg[ g(\tilde{X}_T^n,\mu_T^n) +\int_{t_0}^T \left(
			-H_{\uparrow}(x, \Delta^n_+ u^n) -  H_{\downarrow}(x, -\Delta^n_- u^n)- \tilde{\alpha}_+^n \Delta^n_+ u^n -\tilde{\alpha}^n_- \Delta^n_- u^n 
			\right)(s, \tilde{X}^n_s) +f(\tilde{X}^n_s, \mu^n_s) ds \bigg] \\
			&\leq \E \bigg[ g(\tilde{X}^n_T,\mu^n_T) +\int_{t_0}^T \big(
			-H_{\uparrow}(x, \Delta^n_+ u^n) -L(x,\tilde{\alpha}^n_+) -  H_{\downarrow}(x, -\Delta^n_- u^n) - L(x,-\tilde{\alpha}^n_-) - \tilde{\alpha}_+^n \Delta^n_+ u^n -\tilde{\alpha}^n_- \Delta^n_- u^n\\
			&\qquad+ L(x,\tilde{\alpha}^n_+) +  L(x,-\tilde{\alpha}^n_-) 
			\big)(s, \tilde{X}^n_s) +f(\tilde{X}^n_s, \mu^n_s) ds \bigg],
		\end{align*}
		which gives, using the uniform convexity of $L$ 
		(denoting $|\alpha-\tilde{\alpha}|^2 = |\alpha_+-\tilde{\alpha}_+|^2 + |\alpha_--\tilde{\alpha}_-|^2$ ) 
		\footnote{
			Note that we compare the costs because the Lagrangian is uniformly convex, while the Hamiltonian in general is not; see remark \ref{rem:quadratic}. 
		}
		\be 
		\label{eq:511}
		\E \int_{t_0}^T \frac1C |\alpha^n - \tilde{\alpha}^n|^2(s,\tilde{X}_s) ds \leq 
		J^n( \tilde{\alpha}^n, \mu^n) - \E[u^n(t_0,\xi^n)] .
		\ee 
		We recall that 
		\[
		\E[u^n(t_0,\xi^n)] = J^n(\alpha^n, \mu^n) = \E \bigg[ g(X^n_T,\mu^n_T) +\int_{t_0}^T \left(
		L(x,\alpha_+^n) + L(x,-\alpha_-^n) - L(x,0)
		\right)(s, X^n_s) +f(X^n_s, \mu^n_s) ds \bigg] .
		\]
		On the other hand, evaluating $u$ on $X^n$, a similar argument yields 
		\be 
		\label{eq:521}
		\E \int_t^T \frac1C |\alpha^n - \tilde{\alpha}^n|^2(s,X_s) ds \leq 
		\tilde{J}^n( \alpha^n, \tilde{\mu}^n) - \E[u(t_0,\xi^n)].
		\ee 

		Summing \eqref{eq:511} and \eqref{eq:521}, and then using the monotonicity assumption, we obtain (recalling that 
		$\Law(X^n_s)=\mu_s$ and $\Law(\tilde{X}^n_s)= \tilde{\mu}^n_s$)
		\begin{align*}
			\E \int_{t_0}^T &\frac1C |\alpha^n - \tilde{\alpha}^n|^2(s,X^n_s) + \frac1C |\alpha^n - \tilde{\alpha}^n|^2(s,\tilde{X}^n_s) ds  \\
			&\leq  J^n( \tilde{\alpha}^n, \mu^n) - J^n(\alpha^n, \mu^n) 
			+\tilde{J}^n( \alpha^n, \tilde{\mu}^n) - \tilde{J}^n(\tilde{\alpha}^n, \tilde{\mu}^n)  \\
			&= \E  \bigg[\int_{t_0}^T \bigg(
			f(\tilde{X}^n_s,\mu^n_s) -f(X^n_s,\mu^n_s) 
			+f(X^n_s,\tilde{\mu}^n_s) - f(\tilde{X}^n_s,\tilde{\mu}^n_s)
			+r^n(s, X^n_s) -r^n(s, \tilde{X}^n_s)
			\bigg) ds \\
			& \qquad +g(\tilde{X}^n_T,\mu^n_T) -g(X^n_T,\mu^n_T) 
			+g(X_T,\tilde{\mu}^n_T) - g(\tilde{X}^n_T,\tilde{\mu}^n_T)
			+r^n(T, X^n_T) -r^n(T, \tilde{X}^T_s)
			\bigg] \\
			&\leq
			\int_{t_0}^T ds \int_\T 
			\big( f(x,\mu^n_s) - f(x,\tilde{\mu}^n_s \big) 
			( \tilde{\mu}^n_s -\mu^n_s)(dx)  
			+ \int_{\T} \big(g(x,\mu^n_T) - g(x,\tilde{\mu}^n_T \big) 
			( \tilde{\mu}^n_T -\mu^n_T)(dx)  
			+\frac{C}{n^{\frac13}}\\
			&\leq \frac{C}{n^{\frac13}},
		\end{align*}
		which provides 
		\be
		\label{est:distn}
		\E \int_{t_0}^T  |\alpha^n - \tilde{\alpha}^n|^2(s,X^n_s) + |\alpha^n - \tilde{\alpha}^n|^2(s,\tilde{X}^n_s) ds 
		\leq \frac{C}{n^{\frac13}}
		\ee

		\textbf{Step 3}.
		We can now estimate the distance between $X^n$ and $\tilde{X}^n$. Since $\partial_x  u$ is Lipschitz continuous in $x$, so are $\Delta^n u$ and $\tilde{\alpha}_{\pm}^n$, because $\partial_p H$ is locally Lipschitz in $(x,p)$,  and thus,
		applying \eqref{est:distn} and Jensen's inequality, we  obtain 
		\begin{align*}
			\E \bigg[ \sup_{t_0 \leq s\leq t} |  X^n_t-\tilde{X}^{n}_t | \bigg]
			&\leq    
			\E  \int_{t_0}^t |\alpha^n_+(s, X^n_s) - \tilde{\alpha}^n_+(s,\tilde{X}^n_s)| + 
			|\alpha^n_-(s, X^n_s) - \tilde{\alpha}^n_-(s,\tilde{X}^n_s)|ds \\
			& \leq  + C \E  \int_{t_0}^t | X^n_s -\tilde{X}^n_s| ds 
			+ C \sqrt{\E \int_{t_0}^T |\alpha^n - \tilde{\alpha}^n|^2 (s, X^n_s) ds } \\
			&\leq \frac{C}{n^{\frac16}}  
			+ C  \int_{t_0}^t \E \bigg[ \sup_{t_0 \leq r\leq s} | X^n_s -\tilde{X}^n_s| \bigg] ds 
		\end{align*}
		and hence Gronwall's lemma yields 
		\be 
		\label{eq:74}
		\E \bigg[ \sup_{t_0\leq s\leq T}  | X^n_s- \tilde{X}^{n}_s | \bigg] 
		\leq \frac{C}{n^{\frac16}}.
		\ee 
		This estimate (if the processes start at 0), together with \eqref{rate:W1}, provides \eqref{conv:63}, while with \eqref{conv:65} it proves \eqref{conv:64}.

		\textbf{Step 4}.
		Finally, in order to estimate $|u^n(t,x) -u(t,x)|$, let $J^n(t_0,x,\beta,\mu^n)$ be the cost \eqref{cost:n} where $\mu^n$ is fixed and the dynamics starts at $(t_0,x)$ with control $\beta$, and similarly 
		$\tilde{J}^n(t_0,x,\beta,\tilde{\mu}^n)$ for the cost in \eqref{cost:ntilde}. Clearly $u^n(t_0,x) = \inf_{ \beta} J^n(t_0,x,\beta,\mu^n)$ and $u(t_0,x) = \inf_{ \beta} \tilde{J}^n(t_0,x,\beta,\tilde{\mu}^n)$, the infimum being over open-loop controls. Let $\hat{\beta}$ be an optimal control for $J^n(t_0,x,\beta,\mu^n)$ with corresponding process $X^{n,x}$, with $X^{n,x}_{t_0}=x$. We get, applying the $W_1$-Lipschitz-continuity  of $f$ and $g$ and \eqref{eq:74},
		\begin{align*}
			u(t_0,x) -u^n(t_0,x) &\leq \tilde{J}^n(t_0,x,\hat{\beta}, \tilde{\mu}^n) -J^n(t_0,x,\hat{\beta}, \mu^n)  \\
			&  \leq C \sup_{t_0\leq s\leq T} W_1( \mu^n_s, \tilde{\mu}^n_s)  + 
			\frac{C}{n^{\frac13}}\\
			& \leq \frac{C}{n^{\frac16}},
		\end{align*}
		where we have also used the uniform bound on $r^n$ which appears in $\tilde{J}$. By changing the roles of $\mu^n$ and $\tilde{\mu}^n$ we derive also the opposite inequality, thus 
		\be 
		\label{eq:75}
		|u^n(t_0, x) - u(t_0, x)| \leq \frac{C}{n^{\frac16}},
		\ee 
		where $C$ is independent of $t_0, x, m_0$.

		Finally, to obtain \eqref{conv:62}, recall that $u^n(t_0, x) = U^n(t_0, x, m^n_0)$ and $u(t_0,x) = U(t_0, x, m_0)$. Therefore \eqref{eq:75} and the Lipschitz continuity of $U$ in $m$ provide \eqref{conv:62}.
	\end{proof}

	\section{The case of common noise}
	\label{sec:4}
	We now turn to a case of MFG involving a common noise. As already mentioned the approach here is quite different. Namely we make an extensive use the stability of monotone solutions introduced in \cite{bertucci2021monotone,bertucci2021monotone2}. We postpone recalling the definitions of monotone solutions and we now present the master equations  we are interested in.
	
	The master equation in the continuous state space is
	\be
	\label{masterc}
	\begin{split}
		&-\partial_t U - \sigma\partial_{xx} U + H(x,\partial_x U) 
		+ \int_{\T} \partial_pH(x,\partial_x U(t,y,m)) D_m U(t,x,m; y) m(dy)   \\
		& \qquad -\sigma\int_{\T} \partial_y D_m U(t,x,m;y) m(dy) + \lambda (U - \mathcal{A}^*U(t,x,\mathcal{A}m))= f(m)(x)  \text{ in } (0,T)\times \mathbb{T}\times\mathcal{P}(\mathbb{T})\\
		& U(T,x,m) = g(x,m)  \text{ in } \mathbb{T}\times \mathcal{P}(\mathbb{T}).
	\end{split}
	\ee
	It is very similar to the one at interest in the previous section, except for the presence of terms modeling common noise or common shocks. The common noise is here similar to the one introduced in \cite{bertucci2019some,bertucci2021monotone2}. At random times, of intensity $\lambda > 0$, all the players are affected by the map $\mathcal{A} : \mathcal{P}(\mathbb{T}) \to \mathcal{P}(\mathbb{T})$, and $\mathcal{A}^*$ is the adjoint of $\mathcal{A}$. To fix ideas, mainly for the discretization, we take $\mathcal{A}$ of the form
	\be
	\mathcal{A}m := \int_{\mathbb{T}}K(\cdot,y)m(dy),
	\ee
	where $K : \mathbb{T}^2 \to \R$ is a smooth function which satisfies $K \geq 0, \forall y, \int K(x,y)dx =1$.\\
	
	At the discrete state level, we are hence interested in the following master equation
	\be 
	\label{master:nc}
	\begin{split}
		&-\partial_t U^n(x,m)   +H_{\uparrow}(x,\Delta_+^n U^n(x,m)) +H_{\downarrow} (x,-\Delta_-^n U^n(x,m)) - \sigma\Delta_2^n U^n(x,m) -f(m)(x) \\
		& \qquad -\sum_{y\in S^n} m_y  \left( \frac{-\partial_pH_{\uparrow}(y,(\Delta_+^n U^n(y,m))}{\dn} +\frac{\sigma}{\dn^2} \right) \left( \partial_{m_{y+\dn}} U^n(x,m) - \partial_{m_y} U^n(x,m)\right) \\
		& \qquad - \sum_{y\in S^n} m_y  \left( \frac{\partial_pH_{\downarrow}(y,-\Delta_-^n U^n(y,m))}{\dn} +\frac{\sigma}{\dn^2} \right) \left( \partial_{m_{y-\dn}} U^n(x,m) - \partial_{m_y} U^n(x,m)\right) \\
		&\qquad + \lambda(U - A_n^*U(t,A_nm))=0 \text{ in } (0,T)\times S^n\times \mathcal{P}(S^n),\\
		& U(T,x,m) = g(m)(x)  \text{ in } S^n \times \mathcal{P}(S^n),
	\end{split}
	\ee 
	where $A_n$ is the discretization of $\mathcal{A}$ given by
	\be
	(A_nm)_i = \frac{1}{n}\sum_jK_{i,j}m_j,
	\ee
	where $(K_{i,j})_{1 \leq i,j\leq n}$ is a suitable discretization of $K$ such that for all $1\leq j\leq n$, $\sum_i K_{i,j} = n$. To lighten the notation of the following, we also introduce the operators $F^n,G^n : \mathcal{P}(S^n)\times \R^n$ defined by 
	\be
	\begin{aligned}
		(G^n(m,p))_i &= f(x_i,m)  -H_{\uparrow} (x_i,\Delta_+^n p(x_i)) -H_{\downarrow} (x_i,-\Delta_-^n p(x_i)) + \sigma\Delta_2^n p(x_i)\\
		(F^n(m,p))_i &= -\sum_{j = 1}^n \partial_{p_i}(G^n(m,p))_jm_j,
	\end{aligned}
	\ee
	where $p\in \R^n$ is interpreted as a real function $S^n \to \R$. Using these notations, \eqref{master:nc} can be written
	\be
	-\partial_t U(t,\cdot,m) + (F^n(m,U)\cdot \nabla_m) U + \lambda(U - A_n^*U(t,A_nm)) = G^n(m,U)  \text{ in } (0,T)\times S^n\times \mathcal{P}(S^n).
	\ee
	
	In this part, we are slightly more restrictive with the coupling $f$ and the initial condition $g$. We assume that the following space regularity and stronger monotonicity, which will be in force throughout the section:   
	
	\begin{ass}
		\label{A}
		In addition to the standing assumptions in \S 2.2, we assume that $K$ is a smooth function and 
		for $\alpha \in (0,1]$
		\be
		\begin{aligned}
			\sup_{m } \|f(\cdot,m)\|_{1+ \alpha} + \sup_{m,\nu} \frac{\| f(\cdot,m) - f(\cdot,\nu)\|_{1 + \alpha}}{W_1(m,\nu)} < \infty,\\
			\sup_{m } \|g(\cdot,m)\|_{2+ \alpha} + \sup_{m,\nu} \frac{\| g(\cdot,m) - g(\cdot,\nu)\|_{2 + \alpha}}{W_1(m,\nu)} < \infty.
		\end{aligned}
		\ee
		We also assume that $f$ and $g$ are defined on $\mathcal{M}(\mathbb{T})$ and take values in $\mathcal{C}(\mathbb{T},\mathbb{R})$. We assume that both $f$ and $g$ are differentiable with respect to $m$. The derivative of $f$ with respect to $m$ is given by
		\[
		\frac{\delta f(x,m,y)}{\delta m} = \lim_{h \to 0^+}\frac{f(x,m + h \delta_y) - f(x,m)}{h}.
		\]
		The same holds for $g$. Finally we also assume that there exists $\theta > 0$ such that for all $\nu \in L^2(\T)$
		\begin{equation}\label{eq:strongmonotonicity}
			\int_{\mathbb{T}}\int_{\mathbb{T}}\frac{\delta f(x,m,y)}{\delta m}\nu(x)\nu(y)dxdy \geq \theta \int_\mathbb{T}\left( \int_\mathbb{T}  \frac{\delta f(x,m,y)}{\delta m} \nu(y)dy\right)^2dx.
		\end{equation}
		We also assume that the same holds for $g$.
	\end{ass}

	\begin{rem}
		It is in fact the inequality \eqref{eq:strongmonotonicity} which is important, more than the differentiability of $f$. This assumption is needed in the proof of Proposition \ref{prop:compact}. We could replace this assumption by saying that $f$ and $g$ are the limit of smooth functions satisfying \eqref{eq:strongmonotonicity} uniformly. We prefer to state the assumption in this way so that Proposition \ref{prop:compact} is proved by differentiating the equation, otherwise the proof would be more tedious. 
		
		We remark that we do not assume that the derivatives of $f$ and $g$ are continuous, thus the master equation does not admit classical solutions.   
		
		We finally remark that, even though we extend $f$ and $g$ to the set of all measures, the solutions to the master equations are defined on the set of probability measures only.
	\end{rem}

	\subsection{Monotone solutions of master equations}
	The following is a brief reminder on monotone solutions. We refer to \cite{bertucci2021monotone,bertucci2021monotone2} for details on this concept. We start with the definitions.
	
	\begin{defn}\label{def:monotonec}
		A continuous function $U$, $C^2$  in the space variable $x$, is a monotone solution of \eqref{masterc} if for any measure $\nu \in \mathcal{M}(\mathbb{T})$, $C^2$ function $\phi$ of the space variable and $C^1$ function $\psi$ of the time variable, for any $(t_*,m_*) \in [0,T)\times\mathcal{P}(\mathbb{T})$ point of strict minimum of $(t,m) \to \langle U(t,m) - \phi, m - \nu\rangle - \psi(t)$, the following holds
		\be
		\begin{aligned}
			-\frac{d \psi(t_*)}{dt} + \langle -\sigma\partial_{xx} U + &H(x,\partial_xU) + \lambda(U - \mathcal{A}^*U(\mathcal{A}m_*)), m_* - \nu\rangle \geq \langle f(m_*), m_* - \nu\rangle-\\
			& - \langle U-  \phi, \partial_x(\partial_pH(\cdot,\partial_xU) m_*)\rangle - \sigma\langle \partial_{xx}(U-\phi),m_*\rangle.
		\end{aligned}
		\ee
	\end{defn}
	The same type of definition holds at the discrete level and it is given in this situation by
	\begin{defn}\label{def:monotonef}
		For $n > 0$, a continuous function $U$ is a monotone solution of \eqref{master:nc} if, for any $\nu \in \mathcal{M}(S^n)$, $a \in \mathbb{R}^n$, $C^1$ function of the time $\psi$, and $(t_*,m_*) \in [0,T)\times \mathcal{P}(S^n)$ point of strict minimum of $(t,m) \to \langle U(t,m) - \phi, m - \nu\rangle - \psi(t)$, the following holds
		\be
		\begin{aligned}
			-\frac{d \psi(t_*)}{dt} + \lambda\langle U - A_n^*U(A_n m_*), m_* - \nu\rangle &\geq \langle G^n(m_*,U(t_*,m_*)), m_* - \nu\rangle \\
			&+ \langle F^n(m_*,U(t_*,m_*)),U(t_*,m_*)- a\rangle.
		\end{aligned}
		\ee
	\end{defn}
	
	The previous concept of solution possesses several properties. First we can mention that under the standing assumptions on the monotonicity of $f$ and $g$, there is at most one monotone solution of either \eqref{masterc} or \eqref{master:nc}. Those solutions also enjoy several stability properties, in some sense, this part on the convergence of the discrete state space toward the continuous one is an illustration of this fact. We continue this section with results of existence of such solutions.
	
	\begin{prop}\label{prop:existmaster}
		{Under Asssumption \ref{A},} then there exists a unique monotone solution to both \eqref{masterc} and \eqref{master:nc}. {For the continuous master equation, we say that a solution $U$ is unique if another solution has the same spatial gradient.}
		
		{Moreover, the unique monotone solution $U$ of \eqref{masterc} is such that $\nabla_x U$ and $\Delta_x U$ are continuous on $[0,T]\times \T\times \mathcal{P}(\mathbb{T})$. }
	\end{prop}
	\begin{proof}
		For the continuous state space, a similar result can be found in \cite{bertucci2021monotone2} and for the discrete case, a similar result is in \cite{bertucci2021monotone}. In both cases, the only difference lies in the fact that the Hamiltonian can have a quadratic growth. We leave to the reader these immediate generalizations. {The proof of the last statement is in \cite{bertucci2021monotone2}. }
	\end{proof}
	
	{We refer to Remark \ref{rem:uniq:mono} below for uniqueness result for the function and not only the gradient.  We also remark} that the definition of monotone solution requires some regularity with respect to the space variable in the continuous case whereas, obviously, no such assumption is needed in the finite state case. 
	Even if the setting of Proposition \ref{prop:existmaster} is sufficient to obtain such result, we mention here this difficulty for two reasons. The first one is to explain why this questions shall pop out in the study of the convergence of the master equations, namely because we are going to use this property for the limit equation. Secondly because we believe that this point is of some importance and we shall explain how it can be dealt with in another manner later on in this part. {We refer to \S \ref{sec:4.5} for the details on how to deal with the space regularity of the master equations; in particular, the continuity of the second derivative will be avoided.}  
	
	
	
	{The following is the main result on the convergence of monotone solutions:}
	
	\begin{thm}
		\label{thm:main:mono}
		 Let $U^n$ be the monotone solution to \eqref{master:nc} and $U$ be the monotone solution to \eqref{masterc}. 
		Under Assumption \ref{A}, we have uniform convergence of $U^n$ to $U$ and of their gradients: 
		\be 
		\begin{split}
			&\lim_{n \to \infty} \sup  \{|U^n(t,x,m) - U(t,x,m)|, (t,x,m) \in [0,T]\times S^n \times \mathcal{P}(S^n)\} = 0 \\
			&\lim_{n \to \infty} \sup  \{|\Delta_{\pm}^nU^n(t,x,m) \mp \partial_x U(t,x,m)|, (t,x,m) \in [0,T]\times S^n \times \mathcal{P}(S^n)\} = 0
		\end{split} 
		\ee 
	\end{thm}
	{The theorem is proved in the following sections: 
		the compactness result is in Proposition \ref{prop:compact}, which relies on the main uniform in $n$ estimate on the discrete problem in Theorem \ref{thm:estimate}, while the identification of the limit point is in Theorem \ref{thm:conv:18}. }

	\subsection{A discrete parabolic estimate}
	In this section we present estimates on the semi discrete heat equation, that is discretized in space but not in time. These estimates, in the flavor of parabolic regularity,  are at the same time, fundamental to obtain compactness on the sequence $(U_n)_{n \geq 0}$ of solutions of \eqref{master:nc} {(see next Theorem \ref{thm:estimate})}, quite technical to establish, and not particularly interesting in  themselves  since much more involved results are already well-known in the continuous setting. However, because we could not find sufficiently similar results in the literature, we take some time to explain the proof of such a result.
	
	Our aim is to establish regularity results on the ODE
	\be\label{discretehe} \begin{aligned}
		\dot{u}(t) = \Lambda u + f(t),\\
		u(0) = g,
	\end{aligned}
	\ee
	where $\Lambda$ is defined by
	\[
	\Lambda= n^2\begin{pmatrix}
		-2 & 1 &\dots &\dots  & 1 \\
		1 & -2 & 1 &\dots &\dots \\
		0 & 1 & -2 & 1 &\dots \\ 
		\dots & \dots & \dots & \dots & \dots \\
		1 &\dots &\dots &1 & -2 
	\end{pmatrix} \]
	Clearly $\Lambda$ is a space discretization of the Laplacian operator. We prove the following.
	\begin{thm}\label{thm:discretehe}
		Assume that $g$ is the evaluation of a smooth function on $S^n$. If $f$ is uniformly bounded by a constant $C$, then, the solution $u$ of \eqref{discretehe} satisfies
		
		\be\label{eq:hold1}
		|n^{\alpha-1} \Lambda u_i(t)| \leq C,
		\ee
		for a constant $C$ independent of $n$ and any $\alpha \in (0,\frac{1}{2})$. \\
		If $f$ satisfies for constants $C \geq  0$ and $\alpha  \in (\frac{1}{2},1]$,
		\be\label{eq:hold2}
		n^{\alpha}\left| f_i(t) - f_{i+1}(t)\right| \leq C,
		\ee
		then, the solution $u$ of \eqref{discretehe} satisfies that $\Lambda u$ is bounded by a constant independent of $n$.
	\end{thm}
	\begin{rem}
		The inequality \eqref{eq:hold1} is a sort of $\alpha$-H\"older estimate on the discrete spatial gradient of $u$ and the inequality \eqref{eq:hold2} is a sort of $\alpha$-H\"older estimate on $f$.
	\end{rem}
	\begin{proof}
		Let $\tilde{f} : \mathbb{R} \to \mathbb{R}$ be the function which is $1$-periodic, and the linear interpolation of $f$ on $[0,1]$ with $\tilde{f}(\frac i n) = f_i$. Let us note $(c_k)_{k \in \mathbb{Z}}$ the Fourier exponents of $\tilde{f}$. Because $\tilde{f}$ is continuous, we deduce that it is the sum of its Fourier series, hence
		\be
		f_j = \sum_{k \in \mathbb{Z}}c_k e^{\frac{2ik\pi j}{n}},
		\ee
		Let us define $Q_{kj} = n^{-\frac 1 2} e^{\frac{2i\pi kj}{n}}$ and $\lambda_k = 2(1 - \cos(\frac{2k\pi}{n}))$. The vector $Q_k$ is an eigenvector of $\Lambda$ associated to the eigenvalue $\lambda_k$.  On the other hand, $u$ is given by
		\be
		\begin{aligned}
			u(t) = e^{t \Lambda}g + \int_0^te^{(t-s)\Lambda}f(s)ds.
		\end{aligned}
		\ee
		From which we deduce that 
		\be
		\begin{aligned}
			(\Lambda u(t))_l &= (\Lambda e^{t\Lambda}g)_l - \int_0^t\sum_{k,j} Q_{kl}Q_{kj}\lambda_ke^{-\lambda_k(t-s)}f_j(s)ds\\
			& = (\Lambda e^{t\Lambda }g)_l - \int_0^t\sum_{k,j =1}^n Q_{kl}Q_{kj}\lambda_ke^{-\lambda_k(t-s)}\sum_{p \in \mathbb{Z}}c_p(s) e^{\frac{2i\pi jp}{n}}ds\\
			& =(\Lambda e^{t\Lambda}g)_l - \int_0^t\sum_{k=1}^n e^{\frac{2ikl\pi}{n}}\lambda_ke^{-\lambda_k(t-s)}\sum_{j \in \mathbb{Z}}c_{nj - k}(s) ds\\
		\end{aligned}
		\ee
		We can now observe that estimates on $\Lambda u$ can be obtained by using decay assumptions on the Fourier coefficients of $\tilde{f}$. Unlike for classical regularity ($C^k$ for $k \in \mathbb{N}$), H\"older regularity does not immediately translates into a decay of the Fourier coefficients but rather on their cumulative sum. Namely, if $\tilde{f} \in C^{\alpha}$ for $\alpha \in (0,1)$, {for any $m\in\mathbb{N}$ it holds}
		\be
		\sum_{|k| \in [2^m, 2^{m+1}]} |c_k(t)|\leq C 2^{m\left(\frac{1}{2} - \alpha\right)};
		\ee
		see for instance the first chapter of \cite{katznelson}. In particular, the series is summable if $\alpha > \frac{1}{2}$, which gives the second part of the result. The first part is obtained by remarking that since $\tilde{f}$ is, uniformly in $n$, in $L^{\infty}$, the sequence $(c_k(t))_{k \in \mathbb{Z}}$ is, uniformly in $t$ and $n$, bounded in $\ell^2$. Thus multiplying both sides of the last equation by $n^{\alpha -1}$, we deduce the required result.
	\end{proof}

	\subsection{An estimate on the solution of the master equation without common noise} 
	
	We use the previous estimate to derive an estimate on the solution of \eqref{master:n}, that is \eqref{master:nc} in the case $\lambda= 0$. This estimate is crucial to obtain its counterpart in the case of a common noise. Indeed, as we shall see, thanks to the form of the common noise, we shall reduce the common noise case to the deterministic one. We use the 1-Wasserstein distance $W_1$ on $\mathcal{P}(\T)$ and recall that $U^n$ is evaluated on measures of the form $m=\sum_{j=1}^n m_j \delta_{\frac{j}{n}}$. 
	
	\begin{thm}\label{thm:estimate}
		If $f$ and $g$ satisfy the assumptions of Proposition \ref{prop:existmaster} with $\beta = 1$, then there exists a constant $C$ independent of $n$ such that
		\be 
		\label{estimate:n} 
		|U^n(t,x,m) - U^n(\tilde{t}, \tilde{x}, \tilde{m})| \leq C\left(
		\sqrt[4]{|t-\tilde{t}|} +|x-\tilde{x}| + \sqrt{W_1(m, \tilde{m}) } \right)
		\ee  
		for any $t, \tilde{t} \in [0,T]$, $x,\tilde{x} \in S^n$, $m = \sum_{j=1}^n m_j \delta_{x_j}, \tilde{m}=\sum_{j=1}^n \tilde{m}_j \delta_{x_j} \in \mathcal{P}(S^n)$. Moreover, the discrete gradient of $U^n$, $\Delta_+^n U^n$ also satisfies the same estimate.
	\end{thm} 
	
	\begin{proof}
		\textbf{Step 1.}
		The uniform Lipschitz continuity in space $x$ can be proven exactly as in Lemma \ref{lemma:lip}. 
		
		\textbf{Step 2.}
		To prove the estimate in $m$, fix the initial time $t$ and consider the two solutions of the associated MFG system \eqref{mfgsystem:n} $(u,\mu)$ and $(\tilde{u},\tilde{\mu})$ with $\mu_t=m$, $\tilde{\mu}_t=\tilde{m}$. (Let us omit $n$ in the notation.) Recall that $U^n(t,x,m)= u(t,x)$ and  $U^n(t,x,\tilde{m})=\tilde{u}(t,x)$. Let $\xi$ and $\tilde{\xi}$ be two random variables (the initial conditions) which attain the minimum in the 1-Wasserstein distance, i. e. $\Law(\xi) =m$, $\Law(\tilde{\xi})= \tilde{m}$ and
		\be 
		\label{ini:W}
		\E | \xi-\tilde{\xi}| = W_1( m, \tilde{m}) .
		\ee 
		Consider the optimal feedback control for $(u,m)$:
		$\alpha(s,x) = ( \alpha^+, \alpha^-) = \big(-\partial_pH_{\uparrow}(x,\Delta^n_+ u(s, x) )$ ,$ \partial_pH_{\downarrow}(x, -(\Delta^n_- u(s, x) )\big)$, and similarly let $\tilde{\alpha}$ be the optimal feedback for $\tilde{u}, \tilde{m}$. 
		Let $X$ be the state process driven by the control $\alpha$,
		with $X_t=\xi$, and $\tilde{X}$ be the process driven by the control $\tilde{\alpha}$ with $\tilde{X}_t = \tilde{\xi}$. 
		For $\mu$ fixed and a control $\beta$ (open-loop or feedback), denote by $J(t,\xi,\beta,\mu)$ the cost in \eqref{cost:n} starting at $t$, $\xi$, and similarly $J(t,\tilde{\xi}, \beta, \tilde{\mu})$. 
		
		We compute $u$ on $\tilde{X}$: we have
		\begin{align*}
			\E[ &u(t,\tilde{\xi})] = \E[ u(T, \tilde{X}_T) ] + \E\bigg[ \int_t^T
			\left( -\partial_t u  - \sigma\Delta^n_2 u - \tilde{\alpha}_+ \Delta^n_+ u -\tilde{\alpha}_- \Delta^n_- u \right) (s,\tilde{X}_s) ds \bigg]\\
			&= \E \bigg[ g(\tilde{X}_T,\mu_T) +\int_t^T \left(
			-H_{\uparrow}(x, \Delta^n_+ u) -  H_{\downarrow}(x, -\Delta^n_- u) - \tilde{\alpha}_+ \Delta^n_+ u -\tilde{\alpha}_- \Delta^n_- u 
			\right)(s, \tilde{X}_s) +f(\tilde{X}_s, \mu_s) ds\bigg]. 
		\end{align*}
		A similar computation as in the proof of Theorem \ref{thm:conv5} yields
		\be 
		\label{eq:51}
		\E\bigg[ \int_t^T \frac1C |\alpha - \tilde{\alpha}|^2(s,\tilde{X}_s) ds \bigg]\leq 
		J(t,\tilde{\xi}, \tilde{\alpha}, \mu) - \E[u(t,\tilde{\xi})] .
		\ee 
		Similarly, we get  
		\be 
		\label{eq:52}
		\E \bigg[\int_t^T \frac1C |\alpha - \tilde{\alpha}|^2(s,X_s) ds\bigg] \leq 
		J(t,\xi, \alpha, \tilde{\mu}) - \E[\tilde{u}(t,\xi)] ,
		\ee 
		and we have 
		\begin{align*}
			\E[u(t,\xi)] &=J(t,\xi,\alpha,\mu) \\
			\E[\tilde{u}(t,\tilde{\xi})] &=J(t,\tilde{\xi},\tilde{\alpha},\tilde{\mu}).
		\end{align*}
		Summing \eqref{eq:51} and \eqref{eq:52}, adding and subtracting $\E[u(t,\xi)]$ and $\E[\tilde{u}(t,\tilde{\xi})]$ and then using the monotonicity assumption, we obtain (recalling that 
		$\Law(X_s)=\mu_s$ and $\Law(\tilde{X}_s)= \tilde{\mu}_s$)
		\begin{align*}
			\E \bigg[\int_t^T &\frac1C |\alpha - \tilde{\alpha}|^2(s,X_s) + \frac1C |\alpha - \tilde{\alpha}|^2(s,\tilde{X}_s) ds \bigg] \\
			&\leq \E[ u(t,\xi) - u(t,\tilde{\xi}) +\tilde{u}(t,\tilde{\xi}) - \tilde{u}(t,\xi) ] \\
			&\quad + J(t,\tilde{\xi}, \tilde{\alpha}, \mu) 
			-  J(t,\xi, \alpha, \mu) + J(t,\tilde{\xi}, \tilde{\alpha}, \mu)
			+J(t,\xi,\alpha,\tilde{\mu})-J(t,\tilde{\xi}, \tilde{\alpha}, \tilde{\mu}) \\
			&= \int_{\T} (u-\tilde{u})(0,x) d(m-\tilde{m})(x) 
			+\E  \bigg[\int_t^T \bigg(
			f(\tilde{X}_s,\mu_s) -f(X_s,\mu_s) 
			+f(X_s,\tilde{\mu}_s) - f(\tilde{X}_s,\tilde{\mu}_s)
			\bigg) ds \\
			& \qquad +g(\tilde{X}_T,\mu_T) -g(X_T,\mu_T) 
			+g(X_T,\tilde{\mu}_T) - g(\tilde{X}_T,\tilde{\mu}_T)
			\bigg] \\
			&= \int_{\T} (u-\tilde{u})(0,x) d(m-\tilde{m})(x) \\
			&\quad 
			+
			\int_t^T ds \int_\T 
			\big( f(x,\mu_s) - f(x,\tilde{\mu}_s \big) 
			( \tilde{\mu}_s -\mu_s)(dx)  
			+ \int_{\T} \big(g(x,\mu_T) - g(x,\tilde{\mu}_T \big) 
			( \tilde{\mu}_T -\mu_T)(dx) \\
			&\leq \int_{\T} (u-\tilde{u})(0,x) d(m-\tilde{m})(x)
		\end{align*}
		We now bound the r.h.s using the Lipschitz continuity of $u$ and $\tilde{u}$ to obtain :
		\be
		\label{est:square}
		\E \int_t^T  |\alpha - \tilde{\alpha}|^2(s,X_s) +  |\alpha - \tilde{\alpha}|^2(s,\tilde{X}_s) ds \leq C W_1(m,\tilde{m})
		\ee 
		
		
		\textbf{Step 3.}
		We now use a Lipschitz property on the discrete gradient of $u$, or $\tilde{u}$ (see Step 6 below):
		\be 
		\label{Lip:delta}
		|\Delta^n_{\pm} u (x) - \Delta^n_{\pm} u(\tilde{x}) |\leq C |x-\tilde{x}|.
		\ee 
		If this is true, then applying \eqref{est:square} and Jensen's inequality, we  obtain 
		\begin{align*}
			\E  [|  X_s-\tilde{X}_s | ]
			&\leq 
			\E[|\xi - \tilde{\xi}|  ]+    
			\E  \bigg[\int_t^s |\alpha_+(r, X_r) - \tilde{\alpha}_+(r,\tilde{X}_r)| + 
			|\alpha_-(r, X_r) - \tilde{\alpha}_-(r,\tilde{X}_r)|dr\bigg] \\
			& \leq \E[|\xi - \tilde{\xi}|] + C \E  \bigg[\int_t^s | X_r -\tilde{X}_r| dr \bigg]
			+ C \sqrt{\E \bigg[\int_t^T |\alpha - \tilde{\alpha}|^2 (r, X_r) ds\bigg] } \\
			&\leq C(\E[|\xi - \tilde{\xi}|]+ \sqrt{ W_1(m,\tilde{m} )} + C  \int_t^s \E[| X_r -\tilde{X}_r|]dr 
		\end{align*}
		and thus Gronwall's lemma yields 
		\be 
		\label{root}
		\sup_{t\leq s\leq T} \E [| \tilde{X}_s -  X_s| ]\leq
		C \sqrt{W_1(m,\tilde{m})}.
		\ee
		
		\textbf{Step 4}. We bound the value functions, classically using the characteristics, by 
		\begin{align*} 
			|U^n(t,x,m) - U^n(t,x,\tilde{m}) |&= |u(t,x) -\tilde{u}(t,x)| 
			\leq C  \sup_{t\leq s\leq T} W_1 (\mu_s, \tilde{\mu_s}) \\
			&\leq C \sup_{t\leq s\leq T} \E | \tilde{X}_s -  X_s| 
			\leq C( \E[| \xi - \tilde{\xi}|] )^{\frac{1}{2}}=
			C \sqrt{W_1(m,\tilde{m} )} .
		\end{align*}

		\textbf{Step 5}.
		To prove the estimate in time, let $\tilde{t} >t$ and consider the HJB equation  \eqref{HJB:n} starting at $t$. 
		We recall that $(u,\mu)$ is a MFG solution and thus, for $\mu$ fixed, $u$  represents the value function corresponding to the cost \eqref{cost:n} and we have 
		\[
		u(t,x) = U^n(t,x, m), \qquad u(\tilde{t}, x) = U^n(\tilde{t}, x, \mu^n_{\tilde{t}}),
		\]
		where $\mu_s$ is the Law of the process starting at $(t,m)$ with control given by $\Delta^n u$. 
		Hence the dynamic programmin principle gives 
		\[
		u(t,x) = \E \left[ \int_{t}^{\tilde{t}}  L(X^x_s, 
		\Delta_+^n u (s,X^x_s)) + L(X^x_s,-\Delta_-^n u (s,X^x_s))
		+ f(X^x_s,\mu_s) ds +u(\tilde{t}, X^x_{\tilde{t}}) \right],
		\] 
		where $X^x$ is now the same process as before but conditioned with $X^x_t=x$. Since $u$ is uniformly Lipschitz in space, $\Delta_+^n u$ and $\Delta_-^n u$ are uniformly bounded and we have 
		\begin{align*}
			|u(t,x) - u(\tilde{t}, x)| 
			&\leq |u(t,x) - \E[ u(\tilde{t}, X^x_{\tilde{t}}) ]|
			+ | \E [u(\tilde{t}, X^x_{\tilde{t}})] - u(\tilde{t}, x) | \\ 
			&\leq C(\tilde{t} -t) + C \E [| X^x_{\tilde{t}} - x|] . 
		\end{align*}
		We bound the latter term by using the SDE representation \eqref{dyn:Xn}: 
		\begin{align*} 
			\E [| X^x_{\tilde{t}} - x|^2] &\leq C \E\bigg[ \left| \int_t^{\tilde{t}}  \lambda( \Delta^n_{\pm} u(s, X^x_s), \theta) \nu(d\theta) ds \right|^2 \bigg]
			+C \E \bigg[\left|  \int_t^{\tilde{t}}  \lambda( \Delta^n_{\pm} u(s, X^x_s), \theta) (\mathcal{N}(d\theta, ds - \nu(d\theta) ds) \right|^2\bigg] \\
			&\leq C \E \bigg[\left| \int_t^{\tilde{t}} \dn 
			\frac{ 
				\partial_pH_{\uparrow}(X^x_s,\Delta^n_{+} u(s, X^x_s) ) + \partial_pH_{\downarrow}(X^x_s,-\Delta^n_{-} u(s, X^x_s) ) }{\dn} ds \right|^2 \bigg]
			\\
			& \quad +C \E  \bigg[ \int_t^{\tilde{t}} \left| \lambda( \Delta^n_{\pm} u(s, X^x_s), \theta) \right|^2  \nu(d\theta) ds \bigg]\\
			&\leq C (\tilde{t} -t)^2 
			+C \E  \bigg[ \int_t^{\tilde{t}} \dn^2 \left( \frac{\sigma}{\dn^2} + (\Delta^n_{+} u(s, X^x_s) )_- + \frac{\sigma}{\dn^2} + (\Delta^n_{-} u(s, X^x_s) )_-  \right) ds \bigg]\\
			&\leq C(\tilde{t} -t)^2 +C (\tilde{t} -t)
		\end{align*} 
		and therefore $\E[ | X^x_{\tilde{t}} - x|] \leq C \sqrt{\tilde{t} -t}$, which yields
		\[
		|u(t,x) - u(\tilde{t}, x)| \leq C \sqrt{\tilde{t} -t} .
		\] 
		Similarly we get 
		\[
		W_1 (\mu_{\tilde{t}}, m) \leq C \sqrt{\tilde{t}-t}
		\]
		and hence, applying the H\"older continuity in $m$,
		\begin{align*}
			|U^n(\tilde{t},x,m) - U^n(t,x,m)| 
			&\leq | U^n (\tilde{t}, x, \mu_{\tilde{t}}) -U^n (\tilde{t}, x, m) | 
			+ | u(\tilde{t},x) - u(t,x) | \\
			&\leq C \sqrt{W_1 (\mu_{\tilde{t}}, m)}  + C \sqrt{\tilde{t}-t} \\
			& \leq C(\tilde{t}-t)^{\frac{1}{4}}.
		\end{align*}

		\textbf{Step 6.}
		It remains to prove \eqref{Lip:delta}. This estimate follows from successive uses of Theorem \ref{thm:discretehe} on the discrete HJB equation in the characteristics. Indeed, as we already established the uniform Lipschitz estimate in Lemma \ref{lemma:lip}, passing the Hamiltonian to the right hand side, we deduce from the first part of Theorem \ref{thm:discretehe} a $\alpha$-H\"older type estimate on the spatial gradient of $u$, for $\alpha \in (0,\frac{1}{2})$. Using this new information, we use once again this argument to obtain a higher order regularity on $u$, namely a $\alpha$-H\"older estimate on its spatial gradient for $\alpha \in (\frac 12,1)$. We then use it once more to finally obtain the required boundedness of the discrete Laplacian (uniformly in $n$ of course).

		\textbf{Step 7.}
		The fact that the discrete gradient satisfies the same estimate simply follows from remarking that the two previous steps can be made for the discrete gradient exactly in the same way, and applying  estimate \eqref{root} which we have obtained. Indeed, the equation satisfied by the discrete gradient $v^n = \Delta_+^nu^n$ is given by
		\[
		\begin{aligned}
			-\frac{d}{dt} v^n - \sigma \Delta_2^n v^n &= (\Delta x_n)^{-1}(f(\mu_t^n)(x+ \Delta x_n) - f(\mu_t^n)(x))\\
			&- (\Delta x_n)^{-1}(H_\uparrow(x+ \Delta x_n,v^n(x+\Delta x_n)) - H_\uparrow(x,v^n(x)))\\
			&- (\Delta x_n)^{-1}(H_\downarrow(x+ \Delta x_n,v^n(x)) - H_\downarrow(x,v^n(x-\Delta x_n)))
		\end{aligned}
		\]
		Taking two different solutions $v^n$ and $\tilde{v}^n$ of this equations for two different $\mu$ and $\tilde{\mu}$, we now want to show that an estimate similar to Step 4 can be obtained. This is achieved thanks to a Gr\"onwall estimate on the supremum of $v^n - \tilde{v}^n$. Remark first that from the regularity assumptions we made on $f$, the term involving $f$ will not raise any concern and will yield the require bound with $\sqrt{W_1(m,\tilde{m})}$ thanks to \eqref{root}. Now, concerning the other terms, because the Hamiltonian is sufficiently smooth, we have that $H_\uparrow$ and $H_\downarrow$ are $C^{1,1}$. Hence, we can use the standard maximum principle like estimate (its discrete version to be more precise) to obtain the required estimate. Let us insist on the fact that the estimate \eqref{Lip:delta} is crucial to apply Gr\"onwall's Lemma. The time estimate can then be obtained in a similar fashion as we did in the previous case.
	\end{proof}
	
	\begin{rem}
		The previous result is the only part of this paper in which the dimension $1$ plays a particular role. Indeed, even if it is extremely likely that the estimate proved in Theorem \ref{thm:discretehe} can be generalized to other dimension, it is not proved here. The recent preprint \cite{schauderdiscrete} has been brought to our attention and seems to be a possible answer to this question, and thus could allow to extend this study to higher dimensions.
	\end{rem}
	
	
	\subsection{Compactness results for master equations with common noise}
	\label{sec:4.4}
	In this section, we explain how the previous estimate can be used to gain compactness on the sequence $(U^n)_{n \geq 0}$ of solutions of \eqref{master:nc}. This is mainly achieved by passing the terms modelling common noise to the right side of the equation and by treating them as source terms, thanks to a standard estimate, provided for instance in \cite{bertucci2019some,bertucci2021monotone} in the discrete setting or in \cite{bertucci2021monotone2} in the continuous one. 
	
	\begin{prop}\label{prop:compact}
		Under the assumptions of Proposition \ref{prop:existmaster}, there exists a continuous function $V : [0,T]\times \mathbb{T} \times \mathcal{P}(\mathbb{T})\to \mathbb{R}$ such that, extracting a subsequence if necessary
		\be\label{est:1650}
		\lim_{n \to \infty} \sup  \{|U^n(t,x,m) - V(t,x,m)|, (t,x,m) \in [0,T]\times S^n \times \mathcal{P}(S^n)\} = 0.
		\ee
		Moreover, $\partial_x V$ is continuous on $[0,T]\times \mathbb{T} \times \mathcal{P}(\mathbb{T})$ and 
		\be\label{est:1654}
		\lim_{n \to \infty} \sup  \{|\Delta_{\pm}^nU^n(t,x,m) \mp \partial_xV(t,x,m)|, (t,x,m) \in [0,T]\times S^n \times \mathcal{P}(S^n)\} = 0.
		\ee
	\end{prop}
	\begin{proof}
		We start by establishing why the results of Theorem \ref{thm:estimate} remain true in the case $\lambda > 0$. This relies mainly on an a priori estimate on the solutions of \eqref{master:nc}. This estimate is somehow standard in the literature however, in our context, its introduction is quite tedious. Consider a classical solution $U^n$ of \eqref{master:nc}. We would like to use the computations of the proof of Proposition 1.3 in \cite{bertucci2021monotone}. However the proof is stated for master equations with possibly non-constant mass. This allows in particular to consider without any problem derivatives of $U^n$ with respect to each component of $m \in \mathcal{P}(S^n)$. We do not focus too much on this point, as this does not raise any difficulty here. We recall Appendix B of \cite{bertucci2021monotone} to explain how we can pass from one setting to the other (i.e. from $\mathcal{P}(S^n)$ to $\mathcal{M}(S^n)$). The main idea of the proof of Proposition 1.3 in \cite{bertucci2021monotone} is to show that quantities of the form of \eqref{eq:strongmonotonicity} are propagated for $U^n$ through the master equation. Following the computations of this result we arrived at the following: there exists $C > 0$, such that for all $n >0, (t,m) \in [0,T]\times S^n \times \mathcal{P}(S^n), \xi \in \mathbb{R}^n, \sum_i \xi_i = 0$,
		\[
		\sum_{i = 1}^n\left(\sum_{j=1}^n\partial_{m_i}U^n(t,j\Delta x,m)\xi_j\right)^2 \leq \sum_{i,j = 1}^n\partial_{m_i}U^n(t,j\Delta x,m)\xi_i \xi_j.
		\]
		Hence, we deduce from the Cauchy-Schwarz inequality that
		\[
		\left(\sum_{i = 1}^n\left(\sum_{j=1}^n\partial_{m_i}U^n(t,j\Delta x,m)\xi_j\right)^2\right)^{\frac 12} \leq C \left(\sum_{i =1}^n \xi_i^2\right)^{\frac{1}{2}}
		\]
		This yields a, uniform in $n$, Lipschitz estimate on $U^n(t)$ seen as an operator from $\mathcal{P}(S^n)$ to $\mathbb{R}^n$ when $\mathbb{R}^n$ is equipped with the $\ell_2$ norm and $\mathcal{P}(S^n)$ is equipped with the distance\footnote{This distance can be interpreted as an $L^2(\T)$ distance.} $\tilde{d}(m,m') = \sqrt{n\sum_i (m_i - m_i')^2}$.
		From the properties of the operators $\mathcal{A}$ and $(A_n)_{n> 0}$, we deduce that $m \to \lambda A_n^*U^n(t,\cdot,A_nm)$ is uniformly Lipschitz continuous from $\mathcal{P}(S^n)$ to $\mathbb{R}^n$ when $\mathcal{P}(S^n)$ is equipped with the Wasserstein distance and $\mathbb{R}^n$ with the $\ell_{\infty}$ norm. This is a simple consequence of the regularizing properties of the operators $\mathcal{A}$ and $(A_n)_{n> 0}$.
		
		{It then remains to consider \eqref{master:nc} as of the form of \eqref{master:n}, with the new source term being $f + \lambda A_n^*U^n(A_nm)$.}  We then deduce using Theorem \ref{thm:estimate} that its conclusion is still satisfied here.
		The rest of the proof is now classical. Let us define $\bar{U}^n$ by
		\be
		\begin{aligned}
			&\forall (t,x,m) \in [0,T]\times \mathbb{T} \times \mathcal{P}(\mathbb{T}),\\
			& \bar{U}^n(t,x,m) = \inf \left\{U^n(t,y,\mu) + C |x-y| + C\sqrt{W_1(m,\mu)}\bigg| (y,\mu) \in S^n\times \mathcal{P}(S^n)\right\},
		\end{aligned}
		\ee
		where $C$ is a constant given by the use of Theorem \ref{thm:estimate}. The sequence $(\bar{U}^n)_{n >0}$ satisfies the assumptions of Ascoli-Arzelà Theorem which concludes the proof of \eqref{est:1650}.
		To obtain the additional regularity of $V$, let us remark that, from Theorem \ref{thm:estimate}, we know that the previous argument can be adapted to $(\Delta_{\pm}^nU^n)_{n > 0}$. Hence we also obtained \eqref{est:1654}. 
	\end{proof}

	\subsection{Convergence of the discretized problem} 
	\label{sec:4.5}

	{We now state in which sense any function $V$ given by Proposition \ref{prop:compact} is indeed the unique monotone solution of \eqref{masterc}. 
		The stability of monotone solutions strongly hints that this convergence should happen. The main difficulty relies in the fact that the convergence is here somehow unusual. Moreover, because the uniform convergence of the discrete Laplacian of $(U^n)_{n> 0}$ has not be stated, we need to deal with such terms with care. }
	
	{Let us first remark that in the formulation of monotone solutions of \eqref{masterc}, one only needs the Laplacian of $V$ to make sense against the test measure $\nu$. Indeed the term $\langle \Delta U,m_*\rangle$ appears on the two sides of the inequality. Hence, if $V$ is not sufficiently regular in $x$, one can still test its Laplacian against a measure with a regular density and be able to obtain information in the limit. Hopefully, this will be sufficient to establish that $V$ is indeed a monotone solution. This remark leads us to state the following result which state that $V$ is an approximate monotone solution of the master equation in a certain sense.}
	\begin{lem}\label{lem:conv1}
		Let $V$ be any function given by the Proposition \ref{prop:compact}. For any measure $\nu \in \mathcal{M}(\mathbb{T})\cap W^{2,\infty}(\mathbb{T})$, $C^2$ function $\phi$ of the space variable and $C^1$ function $\psi$ of the time variable, for any $(t_*,m_*) \in [0,T)\times \mathcal{P}(\mathbb{T})$ point of strict minimum of $(t,m) \to \langle V(t,m) - \phi, m - \nu\rangle - \psi(t)$ the following holds
		\be
		\begin{aligned}
			&-\frac{d \psi(t_*)}{dt} + \langle H(x,\nabla_xV) + \lambda(V - \mathcal{A}^*V(\mathcal{A}m_*)), m_* - \nu\rangle \geq \langle f(m_*), m_* - \nu\rangle\\
			& - \langle V-  \phi, \partial_x(\partial_pH(\cdot,\partial_xV) m_*)\rangle + \sigma\langle \partial_{xx} \phi,m_*\rangle - \sigma\langle \partial_{xx} V(t_*,m_*),\nu\rangle.
		\end{aligned}
		\ee
	\end{lem}

	\begin{proof}
		Consider $ \nu,\phi,\psi,t_*,m_*$ as in the statement. For any $n > 0$, consider $\nu^n$ and $\phi^n$ suitable discretizations of $\nu$ and $\phi$. Thanks to Stegall's Lemma \cite{stegall,bishop}, for any $n > 0$, there exists $\delta_n \in \mathbb{R}, a_n \in \R^n$ as small as we want, such that $(t,m) \to \langle U^n(t,m) - \phi^n, m - \nu^n\rangle - \psi(t) + \delta^n t + \langle a_n,m\rangle$ has a strict minimum at $(t_n,m_n)$ on $[0,T] \times \mathcal{P}(S^n)$. Because $U^n$ is a monotone solution of \eqref{master:nc} we obtain that
		\be
		\begin{aligned}
			-\frac{d \psi(t_n)}{dt} - \delta_n + \lambda\langle U^n(m_n) - A_n^*U(A_n m_n), m_n - \nu\rangle &\geq \langle G^n(m_n,U(t_n,m_n)), m_n - \nu^n\rangle \\
			&+ \langle F^n(m_*,U(t_*,m_*)),U(t_*,m_*)- \phi^n - a_n\rangle.
		\end{aligned}
		\ee
		Passing to the limit $n \to \infty$ in the previous inequality yields the required result.
		
	\end{proof}

	On the other hand, defining $\mathcal{P}_M = \{ m \in \mathcal{P}(\mathbb{T}), \|m\|_{2,\infty} \leq M\}$, the unique monotone solution $U$ of \eqref{masterc} satisfies
	\begin{lem}\label{lem:conv2}
		Fix $C > 0$. There exists a function $\omega : \R_+ \to \R_+$ such that $\omega(M) \to 0$ when $M \to \infty$ and for any measure $\nu \in \mathcal{M}(\mathbb{T})$, $\mathcal{C}^{1+\alpha}$ function $\phi$ of the space variable and $\mathcal{C}^1$ function $\psi$ of the time variable, both bounded by $C$, for any $(t_*,m_*)\in [0,T)\times \mathcal{P}_M$ point of strict minimum of $(t,m) \to \langle U(t,m) - \phi, m - \nu\rangle - \psi(t)$ on $[0,T)\times \mathcal{P}_M$, the following holds
		\be\label{eq:1723}
		\begin{aligned}
			&-\frac{d \psi(t_*)}{dt} + \langle H(x,\partial_xU) + \lambda(U - \mathcal{A}^*U(\mathcal{A}m_*)), m_* - \nu\rangle \geq \langle f(m_*), m_* - \nu\rangle\\
			& - \langle U-  \phi, \partial_x(\partial_pH(x,\partial_xU) m_*)\rangle + \sigma\langle \partial_{xx} \phi,m_*\rangle - \sigma\langle \partial_{xx} U(t_*,m_*),\nu\rangle - \omega(M),
		\end{aligned}
		\ee
		where the term $\langle \partial_{xx} \phi,m_*\rangle$ is understood in a weak sense.
	\end{lem}
	In other words, $U$ is almost a solution of \eqref{masterc} on $\mathcal{P}_M$, uniformly in $M$.
	\begin{proof}
		Assume that it is not the case. We first prove the statement for $\mathcal{C}^{2+\alpha}$ functions $\phi$. Reasoning by contradiction, there exists $\kappa > 0$ and a sequence $(\phi_M,\psi_M, t_M,m_M,\nu_M)$ such that $\|\psi_M\|_1 + \|\phi_M \|_{2+\alpha} \leq 2C$, $(t_M,m_M)$ point of strict minimum of $(t,m) \to \langle U(t,m) - \phi_M,m- \nu_M\rangle - \psi_M(t)$ on $[0,T)\times \mathcal{P}_M$ such that
		\be\label{contradict1}
		\begin{aligned}
			&-\frac{d \psi_M(t_*)}{dt} + \langle H(x,\partial_xU) + \lambda(U - \mathcal{A}^*U(\mathcal{A}m_M)), m_M - \nu_M\rangle \leq \langle f(m_M), m_M - \nu_M\rangle\\
			& - \langle U-  \phi, \partial_x(\partial_pH(x,\partial_xU) m_M)\rangle + \sigma\langle \partial_{xx} \phi_M,m_M\rangle - \sigma\langle \partial_{xx} U(t_M,m_M),\nu_M\rangle - \kappa.
		\end{aligned}
		\ee
		Extract a subsequence if necessary and consider the limit point $(\phi_*,\psi_*,\nu_*)$ of the sequence $(\phi_M,\psi_M,\nu_M)_{M \geq 0}$. Using once again Stegall's Lemma, for any $\epsilon > 0$, there exists $\delta \in (-\epsilon,\epsilon)$, $\tilde{\phi}$ such that $\|\tilde{\phi}\|_{2 + \alpha} \leq \epsilon$ and $(t,m) \to \langle U(t,m) - (\phi_* + \tilde{\phi}),m - \nu_*\rangle - (\psi_*(t) + \delta t)$ has a strict minimum on $[0,t_0]\times \mathcal{P}(\mathbb{T})$ at $(t^*,m^*)$. Because $U$ is a monotone solution of \eqref{masterc}
		\be\label{contradict2}
		\begin{aligned}
			&-\frac{d \psi_*(t^*)}{dt} + \delta + \langle H(x,\partial_xU) + \lambda(U - \mathcal{A}^*U(\mathcal{A}m^*)), m^* - \nu_*\rangle \geq \langle f(m^*), m^* - \nu_*\rangle\\
			& - \langle U-  \phi_* - \tilde{\phi}, \partial_x(\partial_pH(x,\partial_xU) m^*)\rangle + \sigma\langle \partial_{xx} (\phi_* + \tilde{\phi}),m^*\rangle - \sigma\langle \partial_{xx}U(t^*,m^*),\nu_*\rangle.
		\end{aligned}
		\ee
		Consider now that $M$ and $\epsilon$ are fixed. Take $\epsilon' > 0, \bar{\phi}$ and $\bar{\delta}$ smaller than $\epsilon$ and consider now a strict minimum $(\bar{t},\bar{m})$ of $(t,m) \to \langle U(t,m) - (\phi_M + \tilde{\phi} + \bar{\phi}),m - \nu_M\rangle -(\psi_M(t) + (\delta + \bar{\delta})t)$. Given that $M$ is large enough, if $\epsilon$ and $\epsilon'$ are sufficiently small, then $(\bar{t},\bar{m})$ is sufficiently close to $(t_M,m_M)$. The uniformity in $M$ large enough comes from the uniform continuity of $U$ and its derivatives and from the convergence of the sequence $(\phi_M,\psi_M,\nu_M)_{M \geq 0}$. On the other hand, from the same argument, given that $M$ is large enough, $(\bar{t},\bar{m})$ is sufficiently close to $(t^*,m^*)$. Using once again the uniform continuity of $U$ and its derivatives, and the convergence of $(\phi_M,\psi_M,\nu_M)_{M \geq 0}$, we obtain a contradiction by comparing \eqref{contradict1} and \eqref{contradict2}.
		
		We then obtain the result for functions $\phi$ in $\mathcal{C}^{1+\alpha}$ in the following way. Assume that the result does not hold for the modulus $\omega$ we just constructed, for a certain function $\phi \in \mathcal{C}^{1+\alpha}\setminus \mathcal{C}^{2+\alpha}$. Consider a sequence $(\phi_n)_{n > 0}$ of smooth functions converging toward $\phi$ in $\mathcal{C}^{1+\alpha}$. Then, thanks to what we just proved, and Stegall's Lemma, we deduce that \eqref{eq:1723} holds at some point $m_n \in \mathcal{P}_M$. Since the minimum of $\Psi:(t,m) \to \langle U(t,m) - \phi, m - \nu\rangle - \psi(t)$ is strict, we deduce that $(m_n)_{n > 0}$ converges toward $m_*$, the second argument at the strict of $\Psi$. By passing to the limit in the monotone solution relation, we then obtain the required result, since $(m_n)_{n>0}$ converges toward $m_*$ in $\mathcal{C}^{1}$ because for all $n> 0, m_n \in \mathcal{P}_M$.
	\end{proof}
	
	The following result hence concludes the proof of Theorem \ref{thm:main:mono}. 
	
	\begin{thm} 
		\label{thm:conv:18}
		A function $V$ given by Proposition \ref{prop:compact} is such that $\nabla_x V= \nabla_x U$ for $U$ a monotone solution of \eqref{masterc}.
	\end{thm}
	\begin{proof}
		Denote by $U$ a monotone solution of \eqref{masterc} and by $V$ a function given by Proposition \ref{prop:compact}. Assume that 
		\be
		\inf_{t \in [0,T],m,m' \in \mathcal{P}(S^n)} \langle U(t,m) - V(t,m'),m-m'\rangle < 0.
		\ee
		Hence, using the uniform continuity of $U$ and $V$, there exists $\kappa > 0$, such that for any $M$ large enough, and any $\gamma > 0$
		\be\label{absurdhyp}
		\inf_{t,s \in [0,T]^2, m \in \mathcal{P}_M,m' \in \mathcal{P}(\mathbb{T})} \langle U(t,m) - V(s,m'),m-m'\rangle + \gamma(t-s)^2 \leq -\kappa,
		\ee
		where $\mathcal{P}_M = \{ m \in \mathcal{P}(\mathbb{T}), \|m\|_{2,\infty} \leq M\}$, which is a compact set. Hence, thanks to Stegall's Lemma, for $\epsilon > 0$ sufficiently small there exists $\delta,\delta' \in ((4T)^{-1}\kappa, (2T)^{-1}\kappa)$, $\phi,\phi' \in C^2$ such that $\|\phi\|_2  + \|\phi'\|_2\leq \epsilon$ and 
		\be
		(t,s,m,m') \to  \langle U(t,m) - V(s,m'),m-m'\rangle + \gamma(t-s)^2 + \langle \phi,m\rangle + \langle \phi',m'\rangle + \delta(T- t) + \delta' (T-s)
		\ee
		has a strict minimum on $[0,T]^2\times \mathcal{P}_M\times \mathcal{P}(\mathbb{T})$ at $(t_*,s_*,m_*,m'_*)$ which is less than $-\frac{\kappa}{2}$. Assume first that $t_*,s_* > 0$.
		Using Lemma \ref{lem:conv1} at this point, we obtain that
		\be
		\begin{aligned}
			&-\delta' -2\gamma (s_*-t_*) + \langle H(x,\partial_xV) + \lambda(V(s_*,m'_*) - \mathcal{A}^*V(s_*,\mathcal{A}m'_*)), m'_* - m_*\rangle \geq \langle f(m'_*), m'_* - m_*\rangle\\
			& - \langle V(s_*,m'_*)-  U(t_*,m_*)+ \phi', \partial_x(\partial_pH(x,\partial_xV)) m'_*)\rangle + \sigma\langle \partial_{xx} (U(t_*,m_*) - \phi') ,m'_*\rangle - \sigma\langle \partial_{xx} V(s_*,m'_*),m_*\rangle.
		\end{aligned}
		\ee
		On the other hand, Lemma \ref{lem:conv2} yields
		\be
		\begin{aligned}
			-\delta - 2\gamma(t_*-s_*) + &\langle H(x,\partial_xU) + \lambda(U(t_*,m_*) - \mathcal{A}^*U(t_*,\mathcal{A}m_*)), m_* - m'_*\rangle \geq \langle f(m_*), m_* - m'_*\rangle\\
			& - \langle U(t_*,m_*)-  V(s_*,m'_*)+ \phi, \partial_x(\partial_pH(x,\partial_xU) m_*)\rangle \\
			&+ \sigma\langle \partial_{xx} (V(s_*,m'_*) - \phi),m_*\rangle - \sigma\langle \partial_{xx} U(t_*,m_*),m'_*\rangle - \omega(M).
		\end{aligned}
		\ee
		Combining the two relations, using the convexity of $H$ and the monotonicity of $f$ yield
		\be
		\begin{aligned}
			-\delta -\delta' &+ \lambda( \langle U(t_*,m_*) - V(s_*,m'_*), m_* - m'_*\rangle - \langle U(t_*,\mathcal{A}m_*) - V(s_*,\mathcal{A}m'_*), \mathcal{A}m_* - \mathcal{A}m'_*\rangle)\\
			& \geq -\omega(M) - \langle \phi', \partial_x(\partial_pH(x,\partial_xV) m'_*)\rangle - \sigma\langle \partial_{xx} \phi',m'_*\rangle - \langle \phi, \partial_x(\partial_pH(x,\partial_xU) m_*)\rangle - \sigma\langle \partial_{xx} \phi,m_*\rangle.
		\end{aligned}
		\ee
		Hence, if $\epsilon$ is chosen small enough, we obtain that 
		\be
		-\frac{\kappa}{2T} \geq -\omega(M),
		\ee
		which is a contradiction if $M$ is large enough.
		
		Consider now the case $t_* = 0$ (the case $s_* = 0$ is similar). In this situation, using the continuity of $U$ and $V$, taking $\gamma$ sufficiently large immediately contradicts \eqref{absurdhyp}. Hence \eqref{absurdhyp} is false and 
		\be
		\inf_{t \in [0,T],m,m' \in \mathcal{P}(S^n)} \langle U(t,m) - V(t,m'),m-m'\rangle \geq 0.
		\ee
		From this we deduce, as in \cite{bertucci2021monotone2} for instance, that $\nabla_x U = \nabla_x V$. 
	\end{proof}
	\begin{rem}
		\label{rem:uniq:mono}
		As usual with monotone solutions, the full equality of two solutions, and not only of their gradient with respect to $x$, can be obtained under additional assumptions. We do not present in full details such directions. The most natural way is to extend (if possible) the domain on which the master equation is set, to positive measure with mass at least one for instance, and not just probability measures, such as in the next section. Or we can assume a stronger monotonicity requirement on $f$, such as 
		\be
		\langle f(\mu) - f(\nu),\mu - \nu \rangle = 0 \Rightarrow \mu = \nu.
		\ee 
		Note that this condition is stronger than the monotonicity assumption in \eqref{eq:strongmonotonicity} since, for example, constant functions are included in \eqref{eq:strongmonotonicity}. 
	\end{rem}

	\subsection{Rate of convergence to a classical solution} 
	\label{sec:4.6} 
	In this section we establish a rate for the convergence of $(U^n)_{n \geq 0}$ toward $U$ when $U$ is a classical solution of \ref{masterc}. To simplify the following discussion we assume that the master equations are set on $\mathcal{M}_2(\mathbb{T})$, the set of positive measures of mass at most $2$ on $\mathbb{T}$. We assume that $f$ and $g$ are indeed defined and monotone on $\mathcal{M}_2(\T)$. We also assume that $f$ and $g$ satisfy the requirements of Proposition \ref{prop:existmaster} where by extension, 
	\be
	W_1(\mu,\nu) = \inf_{\phi}\langle\phi, \mu - \nu\rangle,
	\ee
	where the supremum is taken over $1$-Lipschitz functions $\phi$ such that $\phi(0) = 0$.

	We thus assume that there exists $U$, a classical solution of \eqref{masterc} on $[0,T]\times \mathbb{T}\times\mathcal{M}_2(\mathbb{T})$. By extension we consider the master equation in finite state space \eqref{master:nc} on $[0,T]\times S^n\times \mathcal{M}_2(S^n)$.
	The associated concepts of monotone solution on $\mathcal{M}_2(\mathbb{T})$ or $\mathcal{M}_2(S^n)$ are exactly the same as before except for replacing $\mathcal{P}$ by $\mathcal{M}_2$ in the Definitions \ref{def:monotonec} and \ref{def:monotonef}.
	We proceed as in the case without common noise and consider $V^n$ defined by $V^n(t,x,m) = U(t,x,m)$ on $[0,T] \times S^n \times \mathcal{M}_2(S^n)$. As in the case without common noise, the following holds.

	\begin{prop}
		The function $V^n$ satisfies
		\be\label{master:ncr}
		-\partial_t V^n(t,\cdot,m) + (F^n(m,V^n)\cdot \nabla_m) V^n + \lambda(V^n - A_n^*V^n(t,A_nm)) = G^n(m,V^n) + r^n.
		\ee
		with $|r^n(t,x,m)| \leq C \omega(\frac1n)$, where $\omega$ is a modulus of continuity of $\partial_x U$, $\partial_{xx} U$, $D_m U$, $\partial_y D_m U(\cdot, y)$. In particular, it is a monotone solution of this equation.
	\end{prop} 
	
	We can now state the result about the convergence rate for classical solutions. The proof is different from the one employed in \ref{thm:conv:Un} because a system of characteristics is not available for this kind of common noise. Instead, the argument we use is inspired from the proof of uniqueness of classical solutions in the monotone regime presented in \cite{lions2007cours}. Clearly, this argument applies also to the case without common noise, but gives a worse convergence rate with respect to \eqref{conv:Un}.

	\begin{thm}
		\label{thm20}
		There exists $C>0$ such that
		\be
		\sup_{t \in [0,T], x\in S^n, m \in \mathcal{P}(S^n)} |U^n(t,\cdot,m) - V^n(t,\cdot,m) | \leq C \left(\omega\left(\frac 1n\right)\right)^{\frac 13}.
		\ee
	\end{thm}
	\begin{proof}
		Define $W$ by $W(t,m,m') = \langle U^n(t,m) - V^n(t,m'), m- m'\rangle$ and $\kappa$ by
		\be
		-\kappa = \inf_{t\in[0,T],m,m'\in \mathcal{M}_2(S^n)} W(t,m,m').
		\ee
		Using the fact that $U^n$ is a monotone solution of \eqref{master:nc} and $V^n$ a monotone solution of \eqref{master:ncr}, we arrive at
		
		\be
		-\frac{\kappa}{2T} \geq -\inf_{t,m,m'\in \mathcal{M}_2(S^n)}\langle r^n(t,m), m - m'\rangle \geq -C \omega\left(\frac 1n\right).
		\ee
		Take $t \in[0,T],m \in \mathcal{M}_2(S^n)$ and $z \in \mathcal{M}_2(S^n)$. From the previous estimate we obtain for any $h \in (0,1)$
		\be
		\langle V^n(t,(1-h)m + hz) - U^n(t,m),h(z-m)\rangle \geq -C \omega\left(\frac 1n\right). 
		\ee
		Thus,
		\be
		h\langle V^n(t,m) - U^n(t,m),z-m\rangle + \langle V^n(t,(1-h)m + hz) - V^n(t,m),h(z-m)\rangle \geq -C \omega\left(\frac 1n\right). 
		\ee
		Using the Lipschitz regularity of $U$ (hence of $V^n$)
		\be
		\langle V^n(t,m) - U^n(t,m),z-m\rangle \geq - \frac{C}{h}\omega\left (\frac 1n\right) - Ch.
		\ee
		It follows that
		\be
		\inf_{z \in \mathcal{M}_2(S^n)} \langle V^n(t,m) - U^n(t,m),z-m\rangle \geq -C\sqrt{\omega\left( \frac 1n\right)}.
		\ee
		Arguing similarly for $U^n$ and using the uniform H\"older estimate established, we arrive at
		\be
		\inf_{z \in \mathcal{M}_2(S^n)}  \langle V^n(t,m) - U^n(t,m),m-z\rangle \geq -\frac Ch\omega\left(\frac 1n\right) - C\sqrt{h}.
		\ee
		Hence we deduce, using the previous estimate, that
		\be
		\sup_{z,m \in \mathcal{M}_2(S^n)} \left| \langle V^n(t,m) - U^n(t,m),z-m\rangle  \right| \leq C \left(\omega\left(\frac 1n\right) \right)^{\frac 13}
		\ee
		Now take $m,m' \in \mathcal{P}(\mathbb{T})$. From the previous estimate we obtain by choosing $z = m + m'$
		
		\be
		\left| \langle V^n(t,m) - U^n(t,m),m'\rangle  \right| \leq C \left(\omega\left(\frac 1n\right) \right)^{\frac 13},
		\ee
		from which we obtain the desired result.
	\end{proof}

	\subsection{Another approach to convergence: mollification}
	
	We present here, without many details, another method to prove the convergence of the discrete master equation to the continuous one, with common noise and without assuming that there exists a classical solution of the limit equation. This approach is by  means of mollification of the cost functionals. Indeed, an immediate variation of the mollification procedure on the torus introduced in \cite{cecchindelarue2022} turns out to preserve monotonicity. We only present the main lines of this approach as we believe this idea can be of interest for the reader. We do not provide proofs since other arguments have been given above.
	
	Consider then the function $f(x,m)$ (and the same for $g(x,m)$) and assume as usual that it is monotone, smooth in $x$ and continuous but not regular in $m$. Let $f^{n,\delta}(m,x)$ be the usual mollification by convolution in $m$, for $n$ fixed, in the finite dimensional simplex $\mathcal{P}(S^n)$. Clearly, $\lim_{\delta \rightarrow 0}f^{n,\delta} = f$, for $n$ fixed, and preserves the monotonicity. Let also $f^\epsilon$ be (an immediate variation of) the mollification introduced in \cite{cecchindelarue2022} on $\mathcal{P}(\T)$, which preserves monotonicity and converges to $f$, as $\epsilon\rightarrow 0$. Assume that the Hamiltonian is smooth. Then discrete and continuous master equations related to the coefficients $f^{n,\delta}$ and $f^\epsilon$, respectively, admit classical solutions, which we denote by $U^{n,\delta}$ and $U^{\epsilon}$; the solutions with cost function $f$ are as usual denoted by $U^n$ and $U$. Moreover, we have that 
	\begin{align*}
		\lim_{\delta \rightarrow 0} \sup_{t\in [0,T], x\in S^n, m\in \mathcal{P}(S^n)} &|U^{n,\delta}(t,x,m) - U^n(t,x,m) | = 0, \\
		\lim_{\epsilon \rightarrow 0} \sup_{t\in [0,T], x\in \T, m\in \mathcal{P}(T)} & |U^{\epsilon}(t,x,m) - U(t,x,m) | = 0,
	\end{align*}   
	the first line being, of course, not uniform in $n$ a priori. 
	
	Moreover, the use of Theorem \ref{thm20} gives an estimate of the form 
	\[
	\|U^{n,\delta} - U^{\epsilon}\|_{\infty} \leq C(\epsilon,\delta,n),
	\]
	with $C$ such that 
	\[
	\lim_{\epsilon \to 0} \lim_{n \to \infty} \lim_{\delta \to 0} C(\epsilon,\delta,n) = 0.
	\]
	Hence, using $U^{n,\delta}$ and $U^{\epsilon}$ as intermediates to bound the difference between $U^n$ and $U$, passing to the limit in the order of the previous equation, this method leads to
	\be 
	\lim_{n\rightarrow \infty} \sup_{t\in [0,T], x\in S^n, m\in \mathcal{P}(S^n)} |U^{n}(t,x,m) - U(t,x,m) | =0.
	\ee

	\subsection{A weaker notion of monotone solution}
	We conclude this part on the common noise by indicating another definition of monotone solution which could have been used here and that we believe to have an interest in itself. This concept allows to deal with monotone solution of \eqref{masterc} which are $\mathcal{C}^{1+\alpha}$ for $\alpha \in (0,1)$ with respect to the space variable $x$. This allows to avoid the assumption on the uniform continuity of the space Laplacian of the solution with respect to the measure. The following is very much in the flavor of the work done in \cite{cardaliaguet2021weak} in which this method is introduced to deal with first order MFG.
	
	The definition at interest here is
	\begin{defn}
		A continuous function $U$, uniformly $\mathcal{C}^{1+\alpha}$ in space, is a monotone solution of \eqref{masterc} if there exists a constant $C > 0$ such that for any $\epsilon > 0,\nu \in \mathcal{M}(\mathbb{T})\cap W^{1,\infty}$, $\mathcal{C}^{1+\alpha}$ function $\phi$ of the space variable and $\mathcal{C}^1$ function $\psi$ of the time variable, for any $(t_*,m_*)\in[0,T)\times \mathcal{P}(\mathbb{T})$ point of strict minimum of $(t,m) \to \langle U(t,m) - \phi, m - \nu\rangle - \psi(t) + \epsilon\|m\|_{1,\infty}$, the following holds
		\be
		\begin{aligned}
			&-\frac{d \psi(t_*)}{dt} + \langle -\sigma \partial_{xx} U + H(x,\partial_xU) + \lambda(U - \mathcal{A}^*U(\mathcal{A}m_*)), m_* - \nu\rangle \geq \langle f(m_*), m_* - \nu\rangle\\
			& - \langle U-  \phi, \partial_x(\partial_pH(x,\partial_xU) m_*)\rangle -\sigma \langle \partial_{xx}(U-\phi),m_*\rangle - C\epsilon.
		\end{aligned}
		\ee
	\end{defn}
	The main idea of this definition is to use the fact that, independently of the strategies of the players, the evolution of the underlying repartition of players is continuous in a space of regular repartition of players. The constant $C$ in the previous is directly related to this smoothness.
	
	In some sense, the penalization term in $\epsilon$ in the minimization of the function constrains the minima to be in $W^{1,\infty}$, and this penalization only has a cost $C \epsilon$ because of the smoothness of the evolution of the repartition of players. To derive this formulation, consider a classical solution $U$ of \eqref{masterc}. To lighten notation, we do not come back on the interpretation of the time derivative or of the common noise. Hence, we take $\lambda = 0$ and consider $t\geq 0$ and a point $m_*$ of minimum of $m \to  \langle U(t,m) - \phi, m - \nu\rangle + \epsilon\|m\|_{\infty}$. Denote by $(m_s)_{s \geq 0}$ the solution of 
	\be
	\partial_s m - \sigma\partial_{xx}m - \partial_x(m \partial_pH(x,\partial_x U(t,x,m_s))) = 0 \text{ in } (0,\infty)\times \mathbb{T}
	\ee
	with initial condition $m_0 = m_*$. By definition of $m_*$, for any $s \geq 0$ :
	\be
	\langle U(t,m_*) - \phi, m_* - \nu\rangle + \epsilon\|m_*\|_{\infty} \leq \langle U(t,m_s) - \phi, m_s - \nu\rangle + \epsilon\|m_s\|_{\infty}.
	\ee
	Hence using the fact that $U$ is a classical solution of \eqref{masterc}, we deduce by dividing the previous inequality by $s$ and letting $s\to 0$ that
	\be
	\begin{aligned}
		\langle- \partial_t U,m_* - \nu \rangle + \langle -\sigma\partial_{xx} U + &H(x,\partial_xU) , m_* - \nu\rangle \geq \langle f(m_*), m_* - \nu\rangle - \langle U-  \phi, \partial_x(\partial_pH(x,\partial_xU) m_*)\rangle\\
		&  - \sigma\langle \partial_{xx}(U-\phi),m_*\rangle  + \epsilon \liminf_{s \to 0}s^{-1}(\|m_*\|_{1,\infty} - \|m_s\|_{1,\infty}).
	\end{aligned}
	\ee
	However, since $U$ is $\mathcal{C}^{1+\alpha}$ in $x$, uniformly in $t$ and $m$, there exists $C > 0$ such that, for any $m_*$
	\be
	\liminf_{s \to 0}s^{-1}(\|m_*\|_{1,\infty} - \|m_s\|_{1,\infty}) \geq - C.
	\ee
	This last inequality is a consequence of propagation of $\|\cdot \|_{1,\infty}$ norms by the Fokker-Planck equation
	\be
	\partial_t m - \sigma \Delta m + \text{div}(bm) = 0 \text{ in } (0,\infty)\times \mathbb{T}^d,
	\ee
	for a vector field $b$ in $L^{\infty}((0,\infty),C^{0,\alpha})$. The proof is trivial in dimension $1$ as one can simply integrate the Fokker-Planck equation and use standard parabolic estimates in H\"older norms. In dimension $d \geq 1$ the proof of such a regularity is more involved but it remains true. As this question is far from the main topic of this article we do not detail such a proof here.\\
	
	As a consequence of the previous remark, results of existence and uniqueness of such monotone solutions can be established quite easily following \cite{bertucci2021monotone2,cardaliaguet2021weak}.

	\appendix 
	
	\section{Convergence rate for a diffusion approximation}
	\label{appendixA}
	
	We state here a general result about the approximation of a diffusion (on the torus) with a continuous time Markov chain, which is used several times in the paper and we believe might be of independent interest. The main result is to establish a rate for the convergence of the laws in the Wasserstein distance. Although the approximation of diffusions by Markov chains is certainly not a novelty, we have not been able to find a similar result in the literature.  We rely on an estimate of the the distance between the generators and the semigroups of the processes, which is inspired by the methods of \cite{kolo}, and then on a relation among distances on the space of probability measures.    
	
	\begin{prop} 
		\label{lem:7}
		Let $\alpha\in \mathcal{C}^{\gamma/2, \gamma}([0,T] \times \T)$, for a $\gamma\in (0,1)$,  and let $Y$ satisfy \eqref{dyn} with such $\alpha$ and $Y_0 \sim m_0$. Let also $Y_n$, for any $n$, satisfy \eqref{rate:n} with rates $\alpha_{\pm}$ therein given by 
		$\alpha_+ = \psi_+(\alpha)$ and $\alpha_- = \psi_-(\alpha)$ with $\psi_{\pm}: \R \rightarrow \R$ Lipschitz and such that $\psi_+(\alpha) - \psi_-(\alpha) = \alpha$. Suppose also that 
		$\Law(Y^n_0) = m^n_0$, with $W_1(m^n_0, m_0) \leq \frac1n$.  
		Then 
		\be 
		\label{conv:77}
		\lim_n Y^n = Y    \qquad \mbox{ in law in } \mathcal{D}([0,T], \T).
		\ee
		and  
		\be 
		\label{conv:76}
		\sup_{0\leq t\leq T} W_1( \Law(Y^n_t), \Law(Y_t)) 
		\leq  \frac{C}{n^\frac{\gamma}{2+\gamma}}. 
		\ee
		If in addition 
		$\alpha\in \mathcal{C}^{\gamma/2, 1+\gamma}([0,T] \times \T)$ then 
		\be 
		\label{conv:78}
		\sup_{0\leq t\leq T} W_1( \Law(Y^n_t), \Law(Y_t)) 
		\leq  \frac{C}{n^{\frac13}}. 
		\ee
	\end{prop}
	
	For the application of this result in Section \ref{sec:3}, the controls are to be thought as $\alpha = - \partial_p H(\partial_x u)$, $\alpha_+ = - \partial_p H_\uparrow(\partial_x u)$, $\alpha_- =  \partial_p H_\downarrow(\partial_x u)$. We note that in the case of quadratic Hamiltonian $\alpha_+$ and $\alpha_-$ are simply the positive and negative part of $\alpha$. In particular, the additional assumption  $\alpha\in \mathcal{C}^{\gamma/2, 1+\gamma}([0,T] \times \T)$ of the previous statement always holds under our standing assumptions in \S \ref{S:2.2}.

	\begin{proof}
		We employ the convergence of the generators: denote by $\mathcal{L}^n$ and $\mathcal{L}$ the generators of $Y_n$ and $Y$, respectively. 
		For a function $\phi\in\mathcal{C}^{2+\gamma}(\T)$, with $\gamma\in(0,1]$, we have 
		\begin{align*}
			|\mathcal{L}^n_t &\phi(x) - \mathcal{L}_t \phi(x)| \\
			&= \bigg|
			\left( \frac{\alpha_+(t,x)}{\dn}  +\frac{\sigma}{\dn^2} \right) \left[ \phi(x+\dn) - \phi(x)\right] 
			+\left( \frac{\alpha_-(t,x)}{\dn}  +\frac{\sigma}{\dn^2} \right) \left[ \phi(x-\dn) - \phi(x)\right] \\
			& \qquad - \alpha(t,x) \partial_x \phi(x) 
			- \sigma\partial_{xx} \phi(x) \bigg|
			\\
			&= \Big| \alpha_+(t,x)  \Delta^n_+ \phi(x)  + \alpha_-(t,x)  \Delta^n_- \phi(x)
			+\sigma\Delta_2^n \phi(x) - \alpha(t,x) \partial_x \phi(x) 
			- \sigma\partial_{xx} \phi(x) \Big| \\
			&\leq  |\alpha_+(t,x)| \Big| \Delta^n_+ \phi(x) -\partial_x \phi(x) \Big|  
			+ |\alpha_-(t,x)| \Big| \Delta^n_- \phi(x) +\partial_x \phi(x) \Big|
			+\sigma \Big|\Delta_2^n \phi(x) - \partial_{xx} \phi(x) \Big| \\
			& \leq C\Big( 1+ ||\alpha(t,\cdot)||_\infty\Big)  \frac{||\partial_{xx} \phi||_\infty }{n} 
			+ \sigma \frac{||\partial_{xx} \phi ||_{\gamma}}{n^\gamma },
		\end{align*}
		which implies
		\be 
		\label{eq:78}
		\sup_{0\leq t\leq T} ||\mathcal{L}^n_t \phi - \mathcal{L}_t \phi||_{\infty} \leq 
		\frac{C}{n^\gamma} ( 1 + ||\alpha(\cdot,\cdot)||_\infty ) ||\partial_{xx} \phi ||_\gamma .  
		\ee 
		Convergence of the generators then provides \eqref{conv:77} 
		(applying \cite[Thm. 19.25]{kallemberg}), since $\mathcal{C}^{2+\gamma}(\T)$ is an invariant core of the limiting generator in $\mathcal{C}(\T)$, because of Schauder's estimates. 
		
		To prove \eqref{conv:76}, let $\mathcal{S}^n$ and $\mathcal{S}$ be the semigroups corresponding to $\mathcal{L}^n$ and $\mathcal{L}$: 
		\[
		\mathcal{S}^n_{t,s} \phi (x) = \E[ \phi(X^n_s) | X^n_t=x], \qquad 
		\mathcal{S}_{t,s} \phi (x) = \E[ \phi(X_s) | X_t=x]. 
		\] 
		We recall the usual properties  
		\begin{align*}
			& \mathcal{S}^n_{t,s}  \mathcal{S}^n_{s,r} = \mathcal{S}^n_{t,r} \\ 
			& \frac{d}{ds} \mathcal{S}^n_{t,s}  = \mathcal{S}^n_{t,s}  \mathcal{L}^n_s,
			\qquad \frac{d}{dt} \mathcal{S}^n_{t,s}  = - \mathcal{L}^n_t \mathcal{S}^n_{t,s}  
		\end{align*}
		and similarly for $\mathcal{S}$ and $\mathcal{L}$.  Thus, following Kolokoltsov \cite{kolo},  we can write
		\be 
		\label{133}
		\mathcal{S}^n_{t,s} \phi - \mathcal{S}_{t,s} \phi = 
		\mathcal{S}^n_{t,r} \mathcal{S}_{r,s} |_{r=t}^{r=s} \phi 
		= \int_t^s \frac{d}{dr}\mathcal{S}^n_{t,r} \mathcal{S}_{r,s} \phi dr 
		= \int_t^s \mathcal{S}^n_{t,r} ( \mathcal{L}^n_r -\mathcal{L}_r)    \mathcal{S}_{r,s} \phi dr.
		\ee 
		Thanks to Feynman-Kac formula, the function  $u_s(t,x):=\mathcal{S}_{t,s} \phi(x)$ solves the parabolic backward PDE
		\[
		\begin{cases}
			\partial_t u + \sigma\partial_{xx} u + \alpha(t,x) \partial_x u =0 \qquad &\mbox{ in } [0,s) \times \T, \\
			u(s,x)= \phi(x)  & \mbox{ in } \T.
		\end{cases} 
		\]
		Hence Schauder's estimates, for $\gamma \in (0,1)$, give 
		\be
		\sup_{0\leq t\leq s \leq T}|| \mathcal{S}_{t,s} \phi ||_{2+\gamma} \leq 
		C   ||\phi||_{2+\gamma} 
		\ee
		for a constant $C$ depending on $||\alpha||_{\gamma/2, \gamma}$ and $T$. 
		Moreover, if in addition $\alpha\in \mathcal{C}^{\gamma/2, 1+\gamma}([0,T] \times \T)$ and $\phi \in \mathcal{C}^{2+1}(\T)$, then we have  
		\be
		\sup_{0\leq t\leq s \leq T}|| \mathcal{S}_{t,s} \phi ||_{2+1} \leq 
		C   ||\phi||_{2+1}, 
		\ee
		where $C$ depends on $||\alpha||_{\gamma/2, 1+\gamma}$ and $T$.
		
		From the above estimates and using \eqref{eq:78}, \eqref{133}, and the fact that the transition operator $\mathcal{S}^n$ is a contraction, we obtain, for any $0\leq t< s \leq T$, 
		\begin{align*}
			||\mathcal{S}^n_{t,s} \phi - \mathcal{S}_{t,s} \phi ||_{\infty} 
			&\leq (s-t) \sup_{t\leq r\leq s}||\mathcal{S}^n_{t,r} ( \mathcal{L}^n_r -\mathcal{L}_r)    \mathcal{S}_{r,s} \phi ||_\infty 
			\leq T \sup_{t\leq r\leq s} || ( \mathcal{L}^n_r -\mathcal{L}_r)    \mathcal{S}_{r,s} \phi ||_\infty \\
			&\leq \frac{C}{n^\gamma} \sup_{t\leq r\leq s} || \mathcal{S}_{r,s} \phi ||_{2+\gamma}  
			\leq  \frac{C}{n^\gamma} ||\phi||_{2+\gamma} ,
		\end{align*} 
		whereas, here and below, we fix $\gamma\in(0,1)$ if $\alpha\in \mathcal{C}^{\gamma/2, \gamma}([0,T] \times \T)$, and $\gamma=1$ if $\alpha\in \mathcal{C}^{\gamma/2, 1+\gamma}([0,T] \times \T)$. Therefore, for any 	$0\leq t\leq s \leq T$, 
		\be
		\label{eq:79}
		||  \mathcal{S}^n_{t,s} \phi - \mathcal{S}_{t,s} \phi ||_{\infty} 
		\leq \frac{C}{n^\gamma} ||\phi||_{2+\gamma}. 
		\ee
		We now show that the above estimate, together with the estimate on the initial condition, imply
		\be 
		|\E[  \phi(Y^n_t) ]  - \E[  \phi(Y_t) ]  | \leq  \frac{C}{n^\gamma} ||\phi||_{2+\gamma}
		\qquad \forall \phi \in \mathcal{C}^{2+\gamma}(\T),
		\ee
		uniformly in $t$, 
		that is 
		\be
		\label{eq:81}
		\bigg| \int_\T \phi d( \Law(Y^n_t) - \Law(Y_t) )  \bigg| \leq  \frac{C}{n^\gamma} ||\phi||_{2+\gamma} 
		\qquad \forall \phi \in \mathcal{C}^{2+\gamma}(\T).
		\ee
		Indeed, 
		\begin{align*}
			\E[  \phi(Y^n_t) ]  - \E[  \phi(Y_t) ] &= 
			\int_\T S^n_{0,t} \phi(x) m^n_0(dx) -\int_\T S_{0,t} \phi(x) m_0(dx) \\
			&= \int_\T \big( S^n_{0,t} -S_{0,t} \big)\phi(x) m^n_0(dx) 
			+\int_\T S_{0,t} \phi(x) (m^n_0-m_0)(dx)
		\end{align*}
		and, estimating the first term by \eqref{eq:79} and the second term by $W_1(m^n_0, m_0) \leq \frac1n$ and again by parabolic estimates, we get
		\begin{align*}
			|\E[  \phi(Y^n_t) ]  - \E[  \phi(Y_t) ] | 
			&\leq \frac{C}{n^\gamma} ||\phi||_{2+\gamma} 
			+ || \partial_x (S_{0,t} \phi(\cdot) ) ||_{\infty} W_1(m^n_0, m_0) \\
			&\leq \frac{C}{n^\gamma} ||\phi||_{2+\gamma} + \frac{C}{n} ||\phi||_{2+\gamma}.
		\end{align*}
		
		Finally, the estimates   \eqref{conv:76} and \eqref{conv:78}  are   provided by an estimate of the distance in \eqref{eq:81} (for functions in $\mathcal{C}^{2+\gamma}$) in terms of powers of the Wasserstein distance, which we detail below; see \eqref{eq:139}. 
	\end{proof} 
	
	To complete the above proof, let us give some more details on particular distances on the space of probability measures.  
	We denote by $\zeta_r$, for $r\geq 1$, the \emph{Zolotarev metric} of order $r$, which is defined by 
	\[
	\zeta_r(\mu,\nu) = \sup \left\{ \int_\T \phi d(\mu-\nu) : \phi\in \mathcal{F}_r \right\},
	\]
	where $\mathcal{F}_r$ is the set of $\phi\in\mathcal{C}^{l}(\T)$ with $\phi(0) =\phi'(0) =\dots =\phi^{(l)}(0) =0$, $l$ in the integer such that $l< r \leq l+1$, and 
	$|\phi^{l}(x) - \phi^l(y)| \leq |x-y|^{r-l}$. 
	We note that $\zeta_1=W_1$. In \cite{Rio98}, it is proved in dimension one, for any $r\geq 1$, the relation 
	\be 
	\label{eq:139}
	W_r \leq c_r \zeta_r^{\frac1r}, 
	\ee 
	where $W_r$ is the $r$-Wasserstein distance and $c_r$ is a constant depending just on $r$. For $r=k$ integer, the weaker result 
	\be 
	\label{eq:140}
	W_1 \leq c_k \zeta_k^{\frac1k} 
	\ee 
	is shown to be true in any dimension; see \cite{zolotarev79, zolotarev83} and the more recent  \cite{Bogachev2017}. We also remark that the results \eqref{eq:139} and \eqref{eq:140} hold on the whole space $\R$ or $\R^d$. 
	
	Since we are considering functions on the torus, the set $\mathcal{F}_{2+\gamma}$, for $\gamma\in (0,1]$ is contained in the set of $\phi\in \mathcal{C}^{2+\gamma}$ with $\phi$, $\phi'$ and $\phi''$ bounded by 1 and with $\gamma$-H\"older seminorm bounded by 1, i.e. such that $||\phi||_{2+\gamma} \leq 1$. Thus \eqref{eq:81} implies   
	\[
	\zeta_{2+\gamma}(\Law(Y^n_t), \Law(Y_t) ) \leq \frac{C}{n^\gamma},
	\]
	which yields \eqref{conv:76} and \eqref{conv:78} by means of \eqref{eq:139}, with $r=2+\gamma$ therein.

	
	



\begin{thebibliography}{9}
		\bibitem{achdou2010mean} Achdou, Y., and Capuzzo-Dolcetta, I. (2010). Mean field games: numerical methods. \emph{SIAM Journal on Numerical Analysis}, 48(3), 1136-1162.
		\bibitem{achdou2012mean} Achdou, Y., Camilli, F., and Capuzzo-Dolcetta, I. (2012). Mean field games: numerical methods for the planning problem. \emph{SIAM Journal on Control and Optimization}, 50(1), 77-109.
		\bibitem{achdou2013mean} Achdou, Y., Camilli, F., and Capuzzo-Dolcetta, I. (2013). Mean field games: convergence of a finite difference method. \emph{SIAM Journal on Numerical Analysis}, 51(5), 2585-2612.
		\bibitem{almulla2017two} Almulla, N., Ferreira, R., and Gomes, D. (2017). Two numerical approaches to stationary mean-field games. \emph{Dynamic Games and Applications}, 7(4), 657-682.
		
		
		\bibitem{bayraktar2021} Bayraktar, E., Cecchin, A., Cohen, A., and Delarue, F. (2021). Finite state mean field games with Wright--Fisher common noise. \emph{Journal de Math\'ematiques Pures et Appliqu\'ees}, 147, 98-162.
		
		\bibitem{bayraktarcohen} Bayraktar, E., and Cohen, A. (2018). Analysis of a finite state many player game using its master equation. \emph{SIAM Journal
			on Control and Optimization}, 56(5):3538-3568.
		
		\bibitem{benamou2015augmented} Benamou, J. D., and Carlier, G. (2015). Augmented Lagrangian methods for transport optimization, mean field games and degenerate elliptic equations. \emph{Journal of Optimization Theory and Applications}, 167(1), 1-26.
		
		
		\bibitem{bertucci2021monotone} Bertucci, C. (2021). Monotone solutions for mean field games master equations: finite state space and optimal stopping. \emph{Journal de l'\'Ecole polytechnique--Math\'ematiques}, 8, 1099-1132.
		\bibitem{bertucci2021monotone2} Bertucci, C. (2021). Monotone solutions for mean field games master equations: continuous state space and common noise. \emph{arXiv preprint} arXiv:2107.09531. 
		\bibitem{bertucci2020uzawa} Bertucci, C. (2020). A remark on Uzawa's algorithm and an application to mean field games systems. \emph{ESAIM: Mathematical Modelling and Numerical Analysis}, 54(3), 1053-1071.
		
		\bibitem{bertucci2019some} Bertucci, C., Lasry, J. M., and Lions, P. L. (2019). Some remarks on mean field games. \emph{Communications in Partial Differential Equations}, 44(3), 205-227.
		
		
		\bibitem{briceno2018proximal} Briceno-Arias, L. M., Kalise, D., and Silva, F. J. (2018). Proximal methods for stationary mean field games with local couplings. \emph{SIAM Journal on Control and Optimization}, 56(2), 801-836.
		
		\bibitem{Bogachev2017} Bogachev, V. I., Doledenok, A. N., and Shaposhnikov, S. V. (2017). Weighted Zolotarev Metrics
		and the Kantorovich Metric. 
		\emph{Doklady Mathematics}, 95(2), 113-117. 
		
		\bibitem{pfeiffer} Bonnans, J. F., Liu, P., and Pfeiffer, L. (2022).
		Error estimates of a theta-scheme for second-order mean field games. \emph{arXiv preprint}, arXiv:2212.08128.
		
		\bibitem{cardaliaguet2019master} Cardaliaguet, P., Delarue, F., Lasry, J. M., and Lions, P. L. (2019). The master equation and the convergence problem in mean field games. \emph{Annals of Mathematics Studies}, Princeton University Press.
		
		\bibitem{notescetraro} Cardaliaguet, P., and Porretta, A. (2021). An introduction to mean field game theory. In \emph{Mean Field Games}, chapter
		1, Cetraro, Italy 2019, Cardaliaguet, P., Porretta, A. (Eds.), LNM 2281, 1-148, Springer.
		
		\bibitem{cardaliaguet2021weak} Cardaliaguet, P., and Souganidis, P. (2022). 
		Monotone solutions of the master equation for mean field games with no
		idiosyncratic noise. \emph{SIAM Journal on Mathematical Analysis}, 54(4):4198–4237.
		
		\bibitem{carmona2018probabilistic} Carmona, R., and Delarue, F. (2018). \emph{Probabilistic Theory of Mean Field Games with Applications} I-II. Springer Nature.
		
		
		
		\bibitem{cecchin2019a} Cecchin, A., Dai Pra, P., Fischer, M., and Pelino, G. (2019). On the convergence problem in mean field games: a two state model without uniqueness. \emph{SIAM Journal on Control and Optimization}, 57(4), 2443-2466. 
		
		\bibitem{cecchin2021} Cecchin, A., and Delarue, F. (2022). Selection by vanishing common noise for potential finite state mean field games. \emph{Communications in Partial Differential Equations}, 47 (1), 89-168.  
		
		\bibitem{cecchindelarue2022} Cecchin, A. and Delarue, F. (2022). Weak solutions to the master equation of potential mean field games. arXiv preprint	arXiv:2204.04315.  
		
		\bibitem{cecchinfischer} Cecchin, A., and Fischer, A. (2020). Probabilistic approach to finite state mean field games. \emph{Appl. Math. Optim}, 81(2), 253–300. 
		
		\bibitem{cecchin2019b} Cecchin, A., and Pelino, G. (2019). Convergence, fluctuations and large deviations for finite state mean field games via the master equation. \emph{Stochastic Processes and their Applications}, 129(11), 4510-4555.
		
		
		\bibitem{CCD2014} Chassagneux, J.-F., Crisan, D., Delarue, F (2022). A probabilistic approach to classical solutions of the master equation for large population equilibria. \emph{Memoirs of the AMS}, 280(1379). 
		
		\bibitem{CCD2019} Chassagneux, J.-F., Crisan, D., Delarue, F. (2019). Numerical method for FBSDEs of McKean–Vlasov type. \emph{The Annals of
			Applied Probability}, 29(3), 1640-1684. 
		
		
		\bibitem{delaruecetraro} Delarue, F. (2021). Master equation for finite state mean field games with additive common noise. In \emph{Mean Field Games}, chapter 3,
		Cetraro, Italy 2019, Cardaliaguet, P., Porretta, A. (Eds.), LNM 2281,  203-248, Springer.
		
		\bibitem{schauderdiscrete} Funaki, T. and Sethuraman, S. (2021). Schauder estimate for quasilinear discrete PDEs of parabolic type. \emph{arXiv preprint }	arXiv:2112.13973.  
		
		\bibitem{gangbo2021} Gangbo, W., and M\'esz\'aros, A. R. (2022). Global well-posedness of Master equations for deterministic displacement convex potential mean field games. \emph{Communications on Pure and Applied Mathematics}, 75(12), 2685-2801. 
		
		\bibitem{zhang2021} Gangbo, W., M\'esz\'aros, A. R., Mou, C., and Zhang, J. (2022). Mean field games master equations with non-separable hamiltonians and displacement monotonicity. \emph{Ann. Probab.}, 50(6), 2178-2217.
		
		\bibitem{gomes} Gomes, D.A., Mohr, J., Souza, R.R (2013). Continuous time finite state mean field games. \emph{Appl. Math. Optim.}
		68(1), 99–143.
		
		\bibitem{gomes_book} Gomes, D.A. Pimentel, E.A., and Voskanyan, V. (2016). \emph{Regularity Theory for Mean-Field Game Systems}, Springer, Berlin.
		
		\bibitem{gomes_numerics} Gomes, D.A.,  and Sa\'ude, J. (2021). Numerical methods for finite-state mean-field games satisfying a monotonicity condition. \emph{Appl. Math.  Optim.},  83(1), 51–82.
		
		\bibitem{haeedinkanloo}  Hadikhanloo, S., and Silva, F. J. (2019). Finite mean field games: fictitious play and convergence to a first order continuous mean field game. \emph{Journal de Math\'ematiques Pures et Appliqu\'ees}, 132, 369-397.
		
		
		\bibitem{hcm} Huang, M., Malham\'e, R. P., and Caines, P. E. (2006). Large population stochastic dynamic games: closed-loop McKean-Vlasov systems and the Nash certainty equivalence principle. \emph{Communications in Information and Systems}, 6(3), 221-252.
		
		\bibitem{kallemberg} Kallenberg, O. (2001). \emph{Foundations of Modern Probability}, 2nd edn., Springer, Berlin.
		
		\bibitem{katznelson} Katznelson, Y. (2004). \emph{An introduction to harmonic analysis}. Cambridge University Press.
		
		\bibitem{kolo} Kolokoltsov, V. N.  (2010). \emph{Nonlinear Markov processes and kinetic equations}. Cambridge Tracts in Mathematics, 182. Cambridge University Press. 
		
		\bibitem{kushner} Kushner, H. J., and Dupuis, P. (2001). \emph{Numerical Methods for Stochastic Control Problems in Continuous Time}, volume 24 of Applications of Mathematics. Springer, New
		York, 2nd edition. 
		
		\bibitem{lasry2007mean} Lasry, J. M., and Lions, P. L. (2007). Mean field games. \emph{Japanese journal of mathematics}, 2(1), 229-260.
		
		\bibitem{lauriere} Laurière, M. (2021). Numerical Methods for Mean Field Games and Mean Field Type Control. \emph{arXiv preprint}	arXiv:2106.06231. 
		
		\bibitem{laurierebis} Laurière, M., Perrin, S.,  Geist, M., Pietquin, O. (2022). Learning Mean Field Games: A Survey. \emph{arXiv preprint} arXiv:2205.12944.
		
		\bibitem{lions2007cours} Lions, P. L. (2007). Courses at au Collège de France. Available at www.college-de-france.fr
		
		\bibitem{mou2020} Mou, C., and Zhang, J. (2022). Wellposedness of second order master equations for mean field games with nonsmooth data. \emph{Memoirs of the AMS}, to appear.
		
		\bibitem{mou2022} Mou, C., and Zhang, J. (2022). Mean Field Game Master Equations with Anti-monotonicity Conditions. \emph{arXiv preprint}, arXiv:2201.10762.
		
		
		\bibitem{bishop} Phelps, R. R. (2009). \emph{Convex functions, monotone operators and differentiability} (Vol. 1364). Springer.
		
		\bibitem{Rio98} Rio, E. (1998). Distaces minimales et distances idéales. \emph{C. R. Acad. Sci. Paris}, 326 (1), 1127-1130.
		
		\bibitem{stegall} Stegall, C. (1978). Optimization of functions on certain subsets of Banach spaces. \emph{Mathematische Annalen}, 236(2), 171-176.
		
		
		
		\bibitem{zolotarev79} Zolotarev, V. M. (1979). Properties and relations of certain types of metrics. (Russian) \emph{Studies in mathematical statistics}, 3. Zap. Nauchn. Sem. Leningrad. Otdel. Mat. Inst. Steklov. (LOMI) 87, 18–35. 
		
		\bibitem{zolotarev83} Zolotarev, V. M. (1983). Probability metrics. (Russian) \emph{Teor. Veroyatnost. i Primenen.} 28 (2), 264–287.
		
	\end{thebibliography}
\end{document}